\documentclass[10pt,reqno]{amsart}
\usepackage{amssymb}
\usepackage{amscd}
\usepackage{hyperref}

\newcommand{\func}[1]{\operatorname{#1}}
\newtheorem{theorem}{Theorem}[section]
\newtheorem{corollary}[theorem]{Corollary}
\newtheorem{lemma}[theorem]{Lemma}
\newtheorem{proposition}[theorem]{Proposition}

\newtheorem{definition}[theorem]{Definition}
\newtheorem{remark}[theorem]{Remark}

\numberwithin{equation}{section}
\begin{document}
\title{Modified differentials and basic cohomology for Riemannian foliations}
\author[G.~Habib]{Georges Habib}
\address{Lebanese University \\
Faculty of Sciences II \\
Department of Mathematics\\
P.O. Box 90656 Fanar-Matn \\
Lebanon}
\email[G.~Habib]{ghabib@ul.edu.lb}
\author[K.~Richardson]{Ken Richardson}
\address{Department of Mathematics \\
Texas Christian University \\
Fort Worth, Texas 76129, USA}
\email[K.~Richardson]{k.richardson@tcu.edu}

\begin{abstract}
We define a new version of the exterior derivative on the basic forms of a
Riemannian foliation to obtain a new form of basic cohomology that satisfies
Poincar\'e duality in the transversally orientable case. We use this twisted
basic cohomology to show relationships between curvature, tautness, and
vanishing of the basic Euler characteristic and basic signature.
\end{abstract}

\maketitle

\begin{center}
{\it To the memory of Amine Fawaz}
\end{center} 

\tableofcontents

\section{Introduction}

\subsection{Smooth foliations and basic forms}

Let $(M,\mathcal{F})$ be a smooth, closed manifold of dimension $n$ endowed
with a foliation $\mathcal{F}$ given by an integrable subbundle $L\subset TM$
of rank $p$, with $n=p+q$. The set $\mathcal{F}$ is a partition of $M$ into
immersed submanifolds (\emph{leaves}) such that the transition functions for
the local product neighborhoods (foliation charts) are smooth. The subbundle 
$L=T\mathcal{F}$ is the tangent bundle to the foliation; at each $p\in M$, $%
T_{p}\mathcal{F}=L_{p}$ is the tangent space to the leaf through $p$.

\vspace{0in} 
Basic forms are differential forms on $M$ that locally depend only on the
transverse variables in the foliation charts --- that is, forms $\alpha $
satisfying $X\lrcorner \alpha =X\lrcorner d\alpha =0$ for all $X\in \Gamma
(L)$; the symbol \textquotedblleft $\lrcorner $\textquotedblright\ stands
for interior product. Let $\Omega \left( M,\mathcal{F}\right) \subset \Omega
\left( M\right) $ denote the space of basic forms (see \cite{Re0} for the
original work and the expositions \cite{Re1}, \cite{Mo}, \cite{To}). These
differential forms are preserved by the exterior derivative and are used to
define basic cohomology groups $H_{d}^{\ast }\left( M,\mathcal{F}\right) $
given by 
\begin{equation*}
H_{d}^{k}\left( M,\mathcal{F}\right) =\frac{\ker d_{k}}{\func{image}d_{k-1}}
\end{equation*}%
with%
\begin{equation*}
d_{k}=d:\Omega ^{k}\left( M,\mathcal{F}\right) \rightarrow \Omega
^{k+1}\left( M,\mathcal{F}\right) .
\end{equation*}%
The basic cohomology can be infinite-dimensional, and it can be relatively
trivial. We may also define basic cohomology with values in a foliated
vector bundle; by doing this we gain more topological information about the
leaf space.

Basic cohomology does not necessarily satisfy Poincar\'{e} duality, even if
the foliation is transversally oriented (see \cite{Ca} for the original
example, and see \cite{KTduality}, \cite{To} for a twisted version of Poincar%
\'{e} duality that does hold). 
We emphasize that basic cohomology is a smooth foliation invariant and does
not depend on the choice of metric or any transverse or leafwise geometric
structure. In \cite{ElKacNicol} the authors showed the topological
invariance of basic cohomology.

\subsection{Riemannian foliations and bundle-like metrics}

We assume throughout the paper that the foliation is \emph{Riemannian}; this
means that there is a metric on the local space of leaves --- a
holonomy-invariant transverse metric $g_{Q}$ on the normal bundle $%
Q=TM\diagup L$. The phrase \emph{holonomy-invariant} means the transverse
Lie derivative $\mathcal{L}_{X}g_{Q}$ is zero for all leafwise vector fields 
$X\in \Gamma (L)$. This condition is characterized by the existence of a
unique metric and torsion-free connection $\nabla $ on $Q$ \cite{Re}, \cite%
{Mo}, \cite{To}.

We often assume that the manifold is endowed with the additional structure
of a \emph{bundle-like metric} \cite{Re}, i.e. the metric $g$ on $M$ induces
the metric on $Q\simeq L^{\perp }$. Every Riemannian foliation admits
bundle-like metrics that are compatible with a given $\left( M,\mathcal{F}%
,g_{Q}\right) $ structure. There are many choices, since one may freely
choose the metric along the leaves and also the transverse subbundle $Q$. We
note that a bundle-like metric on a smooth foliation is exactly a metric on
the manifold such that the leaves of the foliation are locally equidistant.
There are topological restrictions to the existence of bundle-like metrics
(and thus Riemannian foliations). Important examples of requirements for the
existence of a Riemannian foliations may be found in \cite{KT1}, \cite{KT2}, 
\cite{Mo}, \cite{To}, \cite{Wo}, \cite{Tar}. One geometric requirement is
that, for any metric on the manifold, the orthogonal projection 
\begin{equation*}
P:L^{2}\left( \Omega \left( M\right) \right) \rightarrow L^{2}\left( \Omega
\left( M,\mathcal{F}\right) \right)
\end{equation*}%
must map the subspace of smooth forms onto the subspace of smooth basic
forms (\cite{PaRi}).

\subsection{The basic Laplacian}

Many researchers have studied basic forms and the basic Laplacian on
Riemannian foliations with bundle-like metrics (see \cite{Re}, \cite{Al}, 
\cite{KT1}, \cite{To}). The basic Laplacian $\Delta _{b}$ for a given
bundle-like metric is a version of the Laplace operator that preserves the
basic forms and that is essentially self-adjoint on the $L^{2}$-closure of
the space of basic forms. The basic Laplacian $\Delta _{b}$ is defined to be%
\begin{equation*}
\Delta _{b}=d\delta _{b}+\delta _{b}d:\Omega \left( M,\mathcal{F}\right)
\rightarrow \Omega \left( M,\mathcal{F}\right) ,
\end{equation*}%
where $\delta _{b}$ is the $L^{2}$-adjoint of the restriction of $d$ to
basic forms: $\delta _{b}=P\delta $ is the ordinary adjoint of $d$ followed
by the orthogonal projection onto the space of basic forms.

The operator $\Delta _{b}$ and its spectrum depend on the choice of the
bundle-like metric and provide invariants of that metric. See \cite{JuRi}, 
\cite{LRi1}, \cite{LRi2}, \cite{PaRi}, \cite{Ri1}, \cite{Ri2} for results.
One may think of this operator as the Laplacian on the space of leaves. This
operator is the appropriate one for physical intuition. For example, the
Laplacian is used in the heat equation, which determines the evolution of
the temperature distribution over a manifold as a function of time. If we
assume that the leaves of the foliation are perfect conductors of heat, then
the basic Laplacian is the appropriate operator that allows one to solve the
heat distribution problem in this situation.

It turns out that the basic Laplacian is the restriction to basic forms of a
second order elliptic operator on all forms, and this operator is not
necessarily symmetric (\cite{PaRi}). Only in special cases is this operator
the same as the ordinary Laplacian. The basic Laplacian $\Delta _{b}$ is
also not the same as the formal Laplacian defined on the local quotient
manifolds of the foliation charts (or on a transversal). This transversal
Laplacian is in general not symmetric on the space of basic forms, but it
does preserve $\Omega \left( M,\mathcal{F}\right) $. The basic heat flow
asymptotics are more complicated than that of the standard heat kernel, but
there is a fair amount known (see \cite{PaRi}, \cite{Ri1}, \cite{Ri2}).

\subsection{The basic adjoint of the exterior derivative and mean curvature}

We assume $\left( M,\mathcal{F},g_{M}\right) $ is a Riemannian foliation of
dimension $p$ and codimension $q$, with bundle-like metric $g_{M}$
compatible with the Riemannian structure $\left( M,\mathcal{F},g_{Q}\right).$
Let 
\begin{equation*}
H=\sum_{i=1}^{p}\pi \left( \nabla _{f_{i}}^{M}f_{i}\right) ,
\end{equation*}%
where $\pi :TM\rightarrow Q$ is the bundle projection and $\left(
f_{i}\right) _{1\leq i\leq p}$ is a local orthonormal frame of $T\mathcal{F}$%
. This is the mean curvature vector field, and its dual one-form is $\kappa
=H^{\flat }$. Let $\kappa _{b}=P\kappa $ be the (smooth) basic projection of
this mean curvature one-form. It turns out that $\kappa _{b}$ is a closed
form whose cohomology class in $H_d^{1}\left( M,\mathcal{F}\right) $ is
independent of the choice of bundle-like metric (see \cite{Al}). Let $\kappa
_{b}\lrcorner $ denote the (pointwise) adjoint of the operator $\kappa
_{b}\wedge $. Clearly, $\kappa _{b}\lrcorner $ depends on the choice of
bundle-like metric $g_{M}$, not simply on the transverse metric $g_{Q}$.

Recall the following expression for $\delta _{b}$ (see \cite{To}, \cite{Al}, 
\cite{PaRi}):%
\begin{eqnarray*}
\delta _{b} &=&P\delta \\
&=&\pm \overline{\ast }d\overline{\ast }+\kappa _{b}\lrcorner \\
&=&\delta _{T}+\kappa _{b}\lrcorner ,
\end{eqnarray*}%
where

\begin{itemize}
\item $\delta _{T}$ is the formal adjoint (with respect to $g_{Q}$) of the
exterior derivative on the transverse local quotients.

\item the pointwise transversal Hodge star operator $\overline{\ast }$ is
defined on all $k$-forms $\gamma $ by%
\begin{equation*}
\overline{\ast }\gamma =\left( -1\right) ^{p\left( q-k\right) }\ast \left(
\gamma \wedge \chi _{\mathcal{F}}\right) ,
\end{equation*}%
with $\chi _{\mathcal{F}}$ being the leafwise volume form, the
characteristic form of the foliation, and $\ast $ being the ordinary Hodge
star operator. Note that $\overline{\ast }^{2}=\left( -1\right) ^{k\left(
q-k\right) }$ on $k$-forms. All that is required for the formula above to be
well-defined is that the Riemannian foliation is transversally oriented. The
formula above is independent of the choice of orientation of the manifold
(equivalently, of the leafwise tangent bundle $T\mathcal{F}$).

\item The sign $\pm $ above only depends on dimensions and the degree of the
basic form.
\end{itemize}

\subsection{Twisted duality for basic cohomology}

Even for transversally oriented Riemannian foliations, Poincar\'{e} duality
does not necessarily hold for basic cohomology.

However, note that $d-\kappa _{b}\wedge $ is also a differential which
defines a cohomology of basic forms. That is, since $d\left( \kappa
_{b}\right) =0$, it follows from the Leibniz rule that $\left( d-\kappa
_{b}\wedge \right) ^{2}=0$ as an operator on forms, and it maps basic forms
to basic forms. On transversally oriented foliations, this differential also
has the property that%
\begin{equation*}
\delta _{b}\overline{\ast }\alpha =\left( -1\right) ^{k+1}\overline{\ast }%
\left( d-\kappa _{b}\wedge \right) \alpha 
\end{equation*}%
on every basic $k$-form $\alpha $ (see \cite{PaRi}). As a result of this
equation and the basic cohomology version of the Hodge theorem (see original
result in \cite{Re} for the case of minimal foliations; see \cite{EKH} for
the more general result and different accounts in \cite{KT1}, \cite{PaRi}),
the transversal Hodge star operator implements an isomorphism between
different kinds of basic cohomology groups (see \cite{KTduality}, \cite{To}):%
\begin{equation*}
H_{d}^{\ast }\left( M,\mathcal{F}\right) \cong H_{d-\kappa _{b}\wedge
}^{q-\ast }\left( M,\mathcal{F}\right) .
\end{equation*}%
This is called \emph{twisted Poincar\'{e} duality}.

\subsection{The basic Dirac operator and spectral rigidity}

We now discuss the construction of the basic Dirac operator (see \cite{DGKY}%
, \cite{GlK}, \cite{PrRi}, \cite{BKR}), a construction which requires a
choice of bundle-like metric. Let $(M,\mathcal{F})$ be a Riemannian manifold
endowed with a Riemannian foliation. Let $E\rightarrow M$ be a foliated
vector bundle (see \cite{KT2}) that is a bundle of $\mathbb{C}\mathrm{l}(Q)$
Clifford modules with compatible connection $\nabla ^{E}$. The \emph{%
transversal Dirac operator} $D_{\mathrm{tr}}$ is the composition of the maps 
\begin{equation*}
\Gamma \left( E\right) \overset{\left( \nabla ^{E}\right) ^{\mathrm{tr}}}{%
\longrightarrow }\Gamma \left( Q^{\ast }\otimes E\right) \overset{\cong }{%
\longrightarrow }\Gamma \left( Q\otimes E\right) \overset{\mathrm{Cliff}}{%
\longrightarrow }\Gamma \left( E\right) ,
\end{equation*}%
where the last map stands for Clifford multiplication, denoted by
\textquotedblleft $\cdot $\textquotedblright , and the operator $\left(
\nabla ^{E}\right) ^{\mathrm{tr}}$ is the projection of $\nabla ^{E}$. The
transversal Dirac operator fixes the basic sections $\Gamma _{b}(E)\subset
\Gamma (E)$ (i.e. $\Gamma _{b}(E)=\{s\in \Gamma (E):\nabla _{X}^{E}s=0$ for
all $X\in \Gamma (L)\}$) but is not symmetric on this subspace. By modifying 
$D_{\mathrm{tr}}$ by a bundle map, we obtain a symmetric and essentially
self-adjoint operator $D_{b}$ on $\Gamma _{b}(E)$. 
We now define 
\begin{eqnarray*}
D_{\mathrm{tr}}~s &=&\sum_{i=1}^{q}e_{i}\cdot \nabla _{e_{i}}^{E}s~, \\
D_{b}s &=&\frac{1}{2}(D_{\mathrm{tr}}+D_{\mathrm{tr}}^{\ast
})s=\sum_{i=1}^{q}e_{i}\cdot \nabla _{e_{i}}^{E}s-\frac{1}{2}\kappa
_{b}^{\sharp }\cdot s~,
\end{eqnarray*}%
where $\{e_{i}\}_{i=1,\cdots ,q}$ is a local orthonormal frame of $Q$. A
direct computation shows that $D_{b}$ preserves the basic sections, is
transversally elliptic, and thus has discrete spectrum (\cite{GlK}, \cite%
{DGKY}, \cite{EK}).

An example of the basic Dirac operator is as follows. Using the bundle $%
\wedge ^{\ast }Q$ as the Clifford bundle with Clifford action $e\cdot
~=e^{\ast }\wedge -e^{\ast }\lrcorner $ in analogy to the ordinary de Rham
operator, we have%
\begin{eqnarray*}
D_{\mathrm{tr}} &=&d+\delta _{T}=d+\delta _{b}-\kappa _{b}\lrcorner :\Omega
^{\mathrm{even}}\left( M,\mathcal{F}\right) \rightarrow \Omega ^{\mathrm{odd}%
}\left( M,\mathcal{F}\right) \\
D_{b} &=&\frac{1}{2}(D_{\mathrm{tr}}+D_{\mathrm{tr}}^{\ast })s=d-\frac{1}{2}%
\kappa _{b}\wedge + \delta _{b}-\frac{1}{2}\kappa _{b}\lrcorner.
\end{eqnarray*}%
One might have incorrectly guessed that $d+\delta _{b}$ is the basic de Rham
operator in analogy to the ordinary de Rham operator, for this operator is
essentially self-adjoint, and the associated basic Laplacian yields basic
Hodge theory that can be used to compute the basic cohomology. The square $%
D_{b}^{2}$ of this operator and the basic Laplacian $\Delta _{b}$ do have
the same principal symbol. In \cite{HabRi}, we showed the invariance of the
spectrum of $D_{b}$ with respect to a change of metric on $M$ in any way
that leaves the transverse metric on the normal bundle intact (this includes
modifying the subbundle $Q\subset TM$, as one must do in order to make the
mean curvature basic, for example). That is,

\begin{theorem}
\label{inv}(In \cite{HabRi}) Let $(M,\mathcal{F})$ be a compact Riemannian
manifold endowed with a Riemannian foliation and basic Clifford bundle $%
E\rightarrow M$. The spectrum of the basic Dirac operator is the same for
every possible choice of bundle-like metric that is associated to the
transverse metric on the quotient bundle $Q$.
\end{theorem}

We emphasize that the basic Dirac operator $D_{b}$ depends on the choice of
bundle-like metric, not merely on the Clifford structure and Riemannian
foliation structure, since both projections $T^{\ast }M\rightarrow Q^{\ast }$
and $P$ as well as $\kappa _{b}\lrcorner $ depend on the leafwise metric. It
is well-known that the eigenvalues of the basic Laplacian $\Delta _{b}$
(closely related to $D_{b}^{2}$) depend on the choice of bundle-like metric;
for example, in \cite[Corollary 3.8]{Ri2}, it is shown that the spectrum of
the basic Laplacian on functions determines the $L^{2}$-norm of the mean
curvature on a transversally oriented foliation of codimension one. 
This is one reason why the invariance of the spectrum of the basic Dirac
operator is a surprise.

\begin{corollary}
\label{basicHarmonicChoiceCorollary}Let $(M,\mathcal{F})$ be a compact
Riemannian manifold endowed with a Riemannian foliation and basic Clifford
bundle $E\rightarrow M$. In calculating the spectrum of the basic Dirac
operator, one may assume the bundle-like metric is chosen so that the mean
curvature form is basic-harmonic.
\end{corollary}

\begin{proof}
\noindent By Theorem \ref{inv}, we may choose the bundle-like metric in any
way that restricts to the given metric on $Q$. In \cite{Do} and \cite%
{MarMinooRuh, Ma}, the researchers showed that there exists such a metric
such that the mean curvature is basic-harmonic. 
\end{proof}

\subsection{Known tautness results}

A Riemannian foliation $\left( M,\mathcal{F}\right) $ is called \textbf{%
minimalizable} or \textbf{geometrically taut} if there exists a Riemannian
metric on $M$ for which all leaves are minimal submanifolds. We will use the
word \textbf{taut} for this property. In \cite{Ghys}, E. Ghys proved that
every Riemannian foliation on a simply connected manifold is taut. X. Masa
showed in \cite{Masa} that a transversally oriented Riemannian foliation of
codimension $q$ is taut if and only if $H_{d}^{q}\left( M,\mathcal{F}\right)
\neq \left\{ 0\right\} $ (see also \cite{KT3}, \cite{KTduality}). Moreover,
J. \`{A}lvarez-Lopez \cite{Al} characterized the tautness by the fact that
the cohomology class of the basic mean curvature form vanishes in $%
H_{d}^{1}\left( M,\mathcal{F}\right) \subseteq H^{1}\left( M\right) $;
therefore, every Riemannian foliation on a manifold with vanishing first
Betti number is taut. In \cite{KTduality}, the results of the F. W. Kamber
and Ph. Tondeur suffice to prove that a foliation is taut if and only if the
basic cohomology groups satisfy Poincar\'{e} duality. In \cite{Noz1} and 
\cite{Noz2}, H. Nozawa showed that the \`{A}lvarez class $\left[ \kappa _{b}%
\right] \in H^{1}\left( M\right) $ is stable with respect to continuous
perturbations of Riemannian foliations, and he also showed that the line
integral of $\kappa _{b}$ over a closed curve is always an algebraic integer
if the fundamental group of $M$ is polycyclic or has polynomial growth.
Under such conditions, a family of Riemannian foliations consists entirely
of taut foliations or of nontaut foliations. As a consequence, he proved
that if $\left( M,\mathcal{F}\right) $ is a codimension two Riemannian
foliation with $\pi _{1}\left( M\right) $ of polynomial growth, then the
foliation is taut.

Certain geometric conditions are also known to force the tautness condition.
For example, J. Hebda showed in \cite{Heb} that if the transversal Ricci
curvature satisfies $\mathrm{Ric}\left(X,X\right) \geq a\left( q-1\right)
\left\vert X\right\vert ^{2}$ for some $a>0$ and all vectors $X\in \Gamma(Q)$%
, then the foliation is taut. In a result by S. D. Jung that was later
extended by the authors (see \cite{Ju} and \cite{HabRi}), if the first
eigenvalue $\lambda $ of the basic spin Dirac operator satisfies equality in
the general eigenvalue bound%
\begin{equation*}
\lambda ^{2}\geq \frac{q}{4\left( q-1\right) }\inf_{M}\left( \mathrm{Scal}%
_{M}-\mathrm{Scal}_{\mathcal{F}}+\left\vert A_{Q}\right\vert ^{2}+\left\vert
T_{\mathcal{F}}\right\vert ^{2}\right) ,
\end{equation*}%
then $\mathcal{F}$ is taut and there exists a transversal Killing spinor.
Here, $A_{Q}$ and $T_{\mathcal{F}}$ denote the O'Neill tensors of the
foliation; see \cite{HabRi} for details.

\subsection{Main results and outline}

In this paper we introduce the new cohomology $\widetilde{H}^{\ast }\left( M,%
\mathcal{F}\right) $ (called the \textbf{twisted basic cohomology}) of basic
forms that uses $\widetilde{d}:=d-\frac{1}{2}\kappa _{b}\wedge $ as a
differential. Recall that the basic de Rham operator is $D_{b}=\widetilde{d}+%
\widetilde{\delta },$ where $\widetilde{\delta }:=\delta _{b}-\frac{1}{2}%
\kappa _{b}\lrcorner .$ We show in Section \ref{Formula Section} that the
corresponding Betti numbers and eigenvalues of the twisted basic Laplacian $%
\widetilde{\Delta }:=\widetilde{d}\,\widetilde{\delta }+\widetilde{\delta }%
\widetilde{d}$ are independent of the choice of a bundle-like metric. In
Theorem \ref{PoincaredualityTheorem} we show that the twisted basic
Laplacian commutes with the transversal Hodge star operator and thus the
twisted basic cohomology satisfies Poincar\'{e} duality. As a corollary, we
deduce that an odd codimension transversally oriented Riemannian foliation
has zero basic Euler characteristic (Corollary \ref%
{EulerZeroInOddCodimCorollary}). In Section \ref{sec:taut}, we prove that
taut foliations give an isomorphism between the new cohomology group and the
basic one. This lets us say that the tautness property is characterized by
the fact that the top-dimensional twisted basic cohomology group is
non-zero. In Section \ref{basicSignatureSection}, we define the basic
signature operator of a Riemannian foliation. Note that the basic signature
was previously defined in \cite{EK2} and used in \cite{HatDj}, but in both
cases the definitions and results hold only for the special case of
Riemannian foliations with minimal leaves.

Using some computations with the Lie derivative, we establish in Section \ref%
{Wein} a Weitzenb\"{o}ck-Bochner formula for the twisted basic Laplacian
(see Proposition \ref{WeitzenbockProposition}), which is more simple that
the corresponding formula for the ordinary basic Laplacian. With the help of
this formula, we deduce various corollaries relating transversal Ricci and
sectional curvature to tautness and basic cohomology. In particular, we
deduce direct proofs of known results of Hebda (see Theorem \ref{Heb}) and
of El Kacimi and others (see Theorem \ref{EK}).

We also study the case of codimension two Riemannian foliations. We prove in
Proposition \ref{codim2Corollary} that for nontaut foliations, the ordinary
basic cohomology satisfies $H_d^{0}\left( M,\mathcal{F}\right) \cong
H_d^{1}\left( M,\mathcal{F}\right) \cong \mathbb{R}$, $H_d^{2}\left( M,%
\mathcal{F}\right) =\left\{ 0\right\}$. Immediate consequences include for
nontaut foliations:

\begin{itemize}
\item The twisted basic cohomology groups are all trivial.

\item The basic Euler characteristic and basic signature are zero (Corollary %
\ref{sigcodim2}).


\item If $\pi _{1}\left( M\right) $ is polycyclic or has polynomial growth,
then the basic Euler characteristic and basic signature are stable with
respect to deformations of $\left( M,\mathcal{F}\right) $ through continuous
families of Riemannian foliations, and the dimensions of all basic
cohomology groups are also stable. See Corollary \ref{deformationCorollary}.
\end{itemize}

To illustrate our results, we treat examples in Section \ref{ExampleSection}.

The second author would like to thank the Mathematisches Forschungsinstitut
Oberwolfach and the Centre de Recerca Matem\`{a}tica (CRM), Barcelona, and
the Department of Mathematics at TCU for hospitality and support during the
preparation of this work.

\section{Modified differentials, Laplacians, and basic cohomology\label%
{Formula Section}}

\vspace{0in} 
Unlike the ordinary and well-studied basic Laplacian, the eigenvalues of $%
\widetilde{\Delta }=D_{b}^{2}$ are invariants of the Riemannian foliation
structure alone and independent of the choice of compatible bundle-like
metric. The operators $\widetilde{d}$ and $\widetilde{\delta }$ have the
following interesting properties.

\begin{lemma}
$\widetilde{\delta }$ is the formal adjoint of $\widetilde{d}.$
\end{lemma}

\begin{proof}
We see%
\begin{equation*}
\left( \widetilde{d}\right) ^{\ast }=\left( d-\frac{1}{2}\kappa _{b}\wedge
\right) ^{\ast }=\delta _{b}-\frac{1}{2}\left( \kappa _{b}\lrcorner \right) =%
\widetilde{\delta },
\end{equation*}%
where the raised $\ast $ denotes formal $L^{2}$-adjoint on the space of
basic forms (not the same as the adjoint on the space of all forms).
\end{proof}

\begin{lemma}
The maps $\widetilde{d}$ and $\widetilde{\delta }$ are differentials; that
is, $\widetilde{d}^{2}=0$, $\widetilde{\delta }^{2}=0$. As a result, $%
\widetilde{d}$ and $\widetilde{\delta }$ commute with $\widetilde{\Delta }%
=D_{b}^{2}$, and $\ker \left( \widetilde{d}+\widetilde{\delta }\right) =\ker
\left( \widetilde{\Delta }\right) $.
\end{lemma}

\begin{proof}
This follows from the fact that $\kappa _{b}$ is a closed one-form \cite{Al}.
\end{proof}

Let $\Omega^{k}\left( M,\mathcal{F}\right) $ denote the space of basic $k$%
-forms (either set of smooth forms or $L^{2}$-completion thereof), let $%
\widetilde{d}^{k}$ and $\widetilde{\delta _{b}}^{k}$ be the restrictions of $%
\widetilde{d}$ and $\widetilde{\delta _{b}}$ to $k$-forms, and let $%
\widetilde{\Delta }^{k}$ denote the restriction of $D_{b}^{2}$ to basic $k$%
-forms.

\begin{proposition}
(Hodge decomposition)\label{Hodge Theorem} We have 
\[\Omega^{k}\left( M,%
\mathcal{F}\right) =\mathrm{image}\left( \widetilde{d}^{k-1}\right) \oplus 
\mathrm{image}\left( \widetilde{\delta _{b}}^{k+1}\right) \oplus \ker \left( 
\widetilde{\Delta }^{k}\right) ,
\]
an $L^{2}$-orthogonal direct sum. Also, $%
\ker \left( \widetilde{\Delta }^{k}\right) $ is finite-dimensional and
consists of smooth forms.
\end{proposition}

\begin{proof}
The proof is very similar to the proof in \cite{PaRi} for the corresponding
fact for the basic Laplacian and in \cite{To}, \cite{KT1} for the basic mean
curvature case. For that reason, we do not include it here.
\end{proof}

We call $\ker \left( \widetilde{\Delta }\right) $ the space of $\widetilde{%
\Delta }$-harmonic forms. In the remainder of this section, we assume that
the foliation is transversally oriented so that the transversal Hodge $%
\overline{\ast }$ operator is well-defined.

\begin{lemma}
(clear) The operator $\overline{\ast }^{2}=\left( -1\right) ^{k\left(
q-k\right) }$ on $k$-forms, and the adjoint of $\overline{\ast }$ is $\left(
-1\right) ^{k\left( q-k\right) }\overline{\ast }$.
\end{lemma}

\begin{lemma}
(in \cite{PaRi}) The basic projection $P$ commutes with $\overline{\ast }$.
\end{lemma}

\begin{lemma}
(in \cite{PaRi}) Given any $\alpha \in \left( N\mathcal{F}\right) ^{\ast }$, 
$\alpha \lrcorner =\left( -1\right) ^{q\left( k+1\right) }\overline{\ast }%
\left( \alpha \wedge \right) \overline{\ast }$ as an operator on basic $k$%
-forms.
\end{lemma}

\begin{lemma}
(in \cite{PaRi}) If $\beta $ is a basic $k$-form, $\delta _{b}\beta =\left(
-1\right) ^{q\left( k+1\right) +1}\overline{\ast }\left( d-\kappa _{b}\wedge
\right) \overline{\ast }\beta.$ 
\end{lemma}

\begin{proposition}
\label{Formulas Proposition}We have the following identities for operators
acting on $\Omega^{k}\left( M,\mathcal{F}\right) $:

\begin{enumerate}
\item $\left( \kappa _{b}\lrcorner \right) \overline{\ast }=\left( -1\right)
^{k}\overline{\ast }\left( \kappa _{b}\wedge \right)$

\item $\overline{\ast }\left( \kappa _{b}\lrcorner \right) =\left( -1\right)
^{k+1}\left( \kappa _{b}\wedge \right) \overline{\ast }$

\item $\delta _{b}\overline{\ast }=\left( -1\right) ^{k+1}\overline{\ast }%
\left( d-\kappa _{b}\wedge \right) $

\item $\overline{\ast }\delta _{b}=\left( -1\right) ^{k}\left( d-\kappa
_{b}\wedge \right) \overline{\ast }$

\item \label{delta-star}$\widetilde{\delta }\overline{\ast }=\left(
-1\right) ^{k+1}\overline{\ast }\widetilde{d}$

\item \label{star-delta}$\overline{\ast }\widetilde{\delta }=\left(
-1\right) ^{k}\widetilde{d}\overline{\ast }$

\item $\overline{\ast }\widetilde{d}=\left( -1\right) ^{k+1}\widetilde{%
\delta }\overline{\ast }$

\item $\widetilde{d}\overline{\ast }=\left( -1\right) ^{k}\overline{\ast }%
\widetilde{\delta }.$
\end{enumerate}
\end{proposition}

\begin{proof}
Acting on basic $k$-forms, we calculate each of the left sides of the
identities above using the lemmas above:%
\begin{eqnarray*}
\left( \kappa _{b}\lrcorner \right) \overline{\ast } &=&\left( -1\right)
^{q\left( q-k+1\right) }\overline{\ast }\left( \kappa _{b}\wedge \right) 
\overline{\ast }^{2} \\
&=&\left( -1\right) ^{q\left( q-k+1\right) +k\left( q-k\right) }\overline{%
\ast }\left( \kappa _{b}\wedge \right) \\
&=&\left( -1\right) ^{k}\overline{\ast }\left( \kappa _{b}\wedge \right).
\end{eqnarray*}%
\begin{eqnarray*}
\overline{\ast }\left( \kappa _{b}\lrcorner \right) &=&\left( -1\right)
^{q\left( k+1\right) }\overline{\ast }^{2}\left( \kappa _{b}\wedge \right) 
\overline{\ast } \\
&=&\left( -1\right) ^{q\left( k+1\right) +\left( q-k+1\right) \left(
k-1\right) }\left( \kappa _{b}\wedge \right) \overline{\ast } \\
&=&\left( -1\right) ^{k+1}\left( \kappa _{b}\wedge \right) \overline{\ast }.
\end{eqnarray*}%
\begin{eqnarray*}
\delta _{b}\overline{\ast } &=&\left( -1\right) ^{q\left( q-k+1\right) +1}%
\overline{\ast }\left( d-\kappa _{b}\wedge \right) \overline{\ast }^{2} \\
&=&\left( -1\right) ^{q\left( q-k+1\right) +1+k\left( q-k\right) }\overline{%
\ast }\left( d-\kappa _{b}\wedge \right) \\
&=&\left( -1\right) ^{k+1}\overline{\ast }\left( d-\kappa _{b}\wedge \right)
.
\end{eqnarray*}%
\begin{eqnarray*}
\overline{\ast }\delta _{b} &=&\left( -1\right) ^{q\left( k+1\right) +1}%
\overline{\ast }^{2}\left( d-\kappa _{b}\wedge \right) \overline{\ast } \\
&=&\left( -1\right) ^{q\left( k+1\right) +1+\left( q-k+1\right) \left(
k-1\right) }\left( d-\kappa _{b}\wedge \right) \overline{\ast } \\
&=&\left( -1\right) ^{k}\left( d-\kappa _{b}\wedge \right) \overline{\ast }.
\end{eqnarray*}%
Putting the results above together, we have%
\begin{eqnarray*}
\widetilde{\delta }\overline{\ast } &=&\left( \delta _{b}-\frac{1}{2}\kappa
_{b}\lrcorner \right) \overline{\ast } \\
&=&\left( -1\right) ^{k+1}\overline{\ast }\left( d-\kappa _{b}\wedge \right)
-\frac{1}{2}\left( -1\right) ^{k}\overline{\ast }\left( \kappa _{b}\wedge
\right) \\
&=&\left( -1\right) ^{k+1}\overline{\ast }\left( d-\frac{1}{2}\kappa
_{b}\wedge \right) \\
&=&\left( -1\right) ^{k+1}\overline{\ast }\widetilde{d},
\end{eqnarray*}%
and%
\begin{eqnarray*}
\overline{\ast }\widetilde{\delta } &=&\overline{\ast }\left( \delta _{b}-%
\frac{1}{2}\kappa _{b}\lrcorner \right) \\
&=&\left( -1\right) ^{k}\left( d-\kappa _{b}\wedge \right) \overline{\ast }-%
\frac{1}{2}\left( -1\right) ^{k+1}\left( \kappa _{b}\wedge \right) \overline{%
\ast } \\
&=&\left( -1\right) ^{k}\left( d-\frac{1}{2}\kappa _{b}\wedge \right) 
\overline{\ast } \\
&=&\left( -1\right) ^{k}\widetilde{d}\overline{\ast }.
\end{eqnarray*}%
Switching sides of the equations in (\ref{delta-star}) and (\ref{star-delta}%
), we obtain%
\begin{eqnarray*}
\overline{\ast }\widetilde{d} &=&\left( -1\right) ^{k+1}\widetilde{\delta }%
\overline{\ast } \\
\widetilde{d}\overline{\ast } &=&\left( -1\right) ^{k}\overline{\ast }%
\widetilde{\delta }.
\end{eqnarray*}
\end{proof}

\begin{definition}
We define the basic $\widetilde{d}$-cohomology $\widetilde{H}^{\ast }\left(
M,\mathcal{F}\right) $ by 
\begin{equation*}
\widetilde{H}^{k}\left( M,\mathcal{F}\right) =\frac{\ker \widetilde{d}^{k}}{%
\mathrm{image}~\widetilde{d}^{k-1}}.
\end{equation*}
\end{definition}

The following proposition follows from standard arguments and the Hodge
theorem (Theorem \ref{Hodge Theorem}).

\begin{proposition}
The finite-dimensional spaces $\widetilde{H}^{k}\left( M,\mathcal{F}%
\right) $ and $\ker $ $\widetilde{\Delta }^{k}=\ker \left( \widetilde{d}+%
\widetilde{\delta }\right) ^{k}$ are naturally isomorphic.
\end{proposition}

We observe that for every choice of bundle-like metric, the differential $%
\widetilde{d}$ changes, and thus the cohomology groups change. However, note
that $\kappa _{b}$ is the only part that changes; for any two bundle-like
metrics $g_{M}$, $g_{M}^{\prime }$ and associated $\kappa _{b}$, $\kappa
_{b}^{\prime }$ compatible with $\left( M,\mathcal{F},g_{Q}\right) $, we
have $\kappa _{b}^{\prime }=\kappa _{b}+dh$ for some basic function $h$ (see 
\cite{Al}). In the proof of the main theorem in \cite{HabRi}, we essentially
showed that the the basic de Rham operator $D_{b}$ is then transformed by $%
D_{b}^{\prime }=e^{h/2}D_{b}e^{-h/2}$. Applying this to our situation, we
see that the $\left( \ker D_{b}^{\prime }\right) =e^{h/2}\ker D_{b}$, and
thus the cohomology groups are the same dimensions, independent of choices.
To see this in our specific situation, note that if $\alpha \in \Omega
^{k}\left( M,\mathcal{F}\right) $ satisfies $\widetilde{d}\alpha =0$, then%
\begin{eqnarray*}
\left( \widetilde{d}\right) ^{\prime }\left( e^{h/2}\alpha \right) &=&\left(
d-\frac{1}{2}\kappa _{b}\wedge -\frac{1}{2}dh\wedge \right) \left(
e^{h/2}\alpha \right) \\
&=&e^{h/2}d\alpha +\frac{1}{2}e^{h/2}dh\wedge \alpha -\frac{e^{h/2}}{2}%
\kappa _{b}\wedge \alpha -\frac{e^{h/2}}{2}dh\wedge \alpha \\
&=&e^{h/2}d\alpha -\frac{e^{h/2}}{2}\kappa _{b}\wedge \alpha =e^{h/2}\left(
d-\frac{1}{2}\kappa _{b}\wedge \right) \alpha =e^{h/2}\widetilde{d}\alpha =0.
\end{eqnarray*}%
Similarly, as in \cite{HabRi} one may show $\ker \left( \widetilde{\delta }%
\right) ^{\prime }=e^{h/2}\ker \left( \widetilde{\delta }\right) $, through
a slightly more difficult computation. Thus, we have

\begin{theorem}
(Conformal invariance of cohomology groups) Given a Riemannian foliation $%
\left( M,\mathcal{F},g_{Q}\right) $ and any two bundle-like metrics $g_{M}$
and $g_{M}^{\prime }$ compatible with $g_{Q}$, the $\widetilde{d}$%
-cohomology groups $\widetilde{H}^{k}\left( M,\mathcal{F}\right) $ are
isomorphic, and that isomorphism is implemented by multiplication by a
positive basic function. Further, the eigenvalues of the corresponding basic
de Rham operators $D_{b}$ and $D_{b}^{\prime }$ are identical, and the
eigenspaces are isomorphic via multiplication by that same positive function.
\end{theorem}

\begin{corollary}
The dimensions of $\widetilde{H}^{k}\left( M,\mathcal{F}\right) $ and the
eigenvalues of $D_{b}$ (and thus of $\widetilde{\Delta }=D_{b}^{2}$) are
invariants of the Riemannian foliation structure $\left( M,\mathcal{F}%
,g_{Q}\right) $, independent of choice of compatible bundle-like metric $%
g_{M}$.
\end{corollary}

\begin{corollary}
\label{indepTwBasicCohomCor}The dimensions of $\widetilde{H}^{k}\left( M,%
\mathcal{F}\right) $ are independent of the choice of the bundle-like metric
and independent of the transverse Riemannian foliation structure.
\end{corollary}

\begin{proof}
By \cite{Al}, the basic components of the mean curvature forms for two
different bundle-like metrics differ by an exact basic one-form $\kappa
_{b}^{\prime }=\kappa _{b}+dh$. Since 
\begin{equation*}
\left(\widetilde{d}\right)^{\prime }=e^{h/2}\widetilde{d}e^{-h/2}
\end{equation*}%
by the computation above, the twisted basic cohomology groups corresponding
to the different metrics are conjugate.
\end{proof}

\section{Poincar\'{e} duality and consequences}

\begin{theorem}
\label{PoincaredualityTheorem}(Poincar\'{e} duality for $\widetilde{d}$%
-cohomology) Suppose that the Riemannian foliation $\left( M,\mathcal{F}%
,g_{Q}\right) $ is transversally oriented and is endowed with a bundle-like
metric. For each $k$ such that $0\leq k\leq q$ and any compatible choice of
bundle-like metric, the map $\overline{\ast }:\Omega ^{k}\left( M,\mathcal{F}%
\right) \rightarrow \Omega ^{q-k}\left( M,\mathcal{F}\right) $ induces an
isomorphism on the $\widetilde{d}$-cohomology. Moreover, $\overline{\ast }$
maps the $\ker \widetilde{\Delta }^{k}$ isomorphically onto $\ker \widetilde{%
\Delta }^{q-k}$, and it maps the $\lambda $-eigenspace of $\widetilde{\Delta 
}^{k}$ isomorphically onto the $\lambda $-eigenspace of $\widetilde{\Delta }%
^{q-k}$, for all $\lambda \geq 0$.
\end{theorem}

\begin{proof}
Acting on basic forms of degree $k$, we use 
\begin{eqnarray*}
\overline{\ast }\widetilde{\Delta } &=&\overline{\ast }\widetilde{d}%
\widetilde{\delta }+\overline{\ast }\widetilde{\delta }\widetilde{d} \\
&=&\left( -1\right) ^{k}\widetilde{\delta }\overline{\ast }\widetilde{\delta 
}+\left( -1\right) ^{k+1}\widetilde{d}\overline{\ast }\widetilde{d} \\
&=&\left( -1\right) ^{k}\widetilde{\delta }\left( -1\right) ^{k}\widetilde{d}%
\overline{\ast }+\left( -1\right) ^{k+1}\widetilde{d}\left( -1\right) ^{k+1}%
\widetilde{\delta }\overline{\ast } \\
&=&\left( \widetilde{\delta }\widetilde{d}+\widetilde{d}\widetilde{\delta }%
\right) \overline{\ast }=\widetilde{\Delta }\overline{\ast}.
\end{eqnarray*}%
Since $\overline{\ast }$ commutes with $\widetilde{\Delta }$, it maps
eigenspaces of $\widetilde{\Delta }$ to themselves. By the Hodge theorem,
the result follows.
\end{proof}

This resolves the problem of the failure of Poincar\'{e} duality to hold for
standard basic cohomology (see \cite{KTduality}, \cite{To}).

\begin{corollary}
Let $\left( M,\mathcal{F}\right) $ be a smooth transversally oriented
foliation of odd codimension that admits a transverse Riemannian structure.
Then the Euler characteristic associated to the $\widetilde{H}^{\ast }\left(
M,\mathcal{F}\right)$ vanishes.
\end{corollary}


\begin{corollary}
\label{EulerZeroInOddCodimCorollary}Let $\left( M,\mathcal{F}\right) $ be a
smooth transversally oriented foliation of odd codimension that admits a
transverse Riemannian structure. Then the Euler characteristic associated to
the ordinary basic cohomology $H_d^{\ast }\left( M,\mathcal{F}\right) $
vanishes.
\end{corollary}

\begin{proof}
The basic Euler characteristic is the basic index of the operator $%
D_{0}=d+\delta _{B}:\Omega^{\mathrm{even}}\left( M,\mathcal{F}\right)
\rightarrow \Omega^{\mathrm{odd}}\left( M,\mathcal{F}\right) $. See \cite%
{BKR}, \cite{DGKY}, \cite{BePaRi}, \cite{EK} for information on the basic
index and basic Euler characteristic. The crucial property for us is that
the basic index of $D_{0}$ is a Fredholm index and is invariant under
perturbations of the operator through transversally elliptic operators that
map the basic forms to themselves. In particular, the family of operators $%
D_{t}=d+\delta _{b}-\frac{t}{2}\kappa _{b}\lrcorner -\frac{t}{2}\kappa
_{b}\wedge $ for $0\leq t\leq 1$ meets that criteria, and $D_{1}=D_{b}$ is
the basic de Rham operator $D_{b}:\Omega^{\mathrm{even}}\left( M,\mathcal{F}%
\right) \rightarrow \Omega^{\mathrm{odd}}\left( M,\mathcal{F}\right) $.
Thus, the basic Euler characteristic of the basic cohomology complex is the
same as the basic Euler characteristic of the $\widetilde{d}$-cohomology
complex. The result follows from the previous corollary.
\end{proof}

Using this result, we give another proof for the well-known theorem \cite%
{Sul}:

\begin{theorem}
Let $\mathcal{F}$ be a transversally Riemannian oriented foliation of
codimension $1$ on a closed manifold $M$. Then $\mathcal{F}$ is taut.
\end{theorem}

\begin{proof}
Since the codimension is odd, the basic Euler characteristic is zero. Thus,
we get that $\mathrm{dim}\, H^1_B(M,\mathcal{F})=1$ and the foliation is
taut by the result of Masa.
\end{proof}

\section{Tautness Theorem}

\label{sec:taut} 
As we mentionned in the introduction, the author showed in \cite{Al} that
the cohomology class $\left[ \kappa _{b}\right] \in H_{d}^{1}\left( M,%
\mathcal{F}\right) $ is an invariant of the Riemannian foliation structure
and independent of the choice of bundle-like metric, and the foliation is
taut if and only if $\left[ \kappa _{b}\right] =0$. If in addition the
foliation is transversally oriented, this condition is equivalent to $%
H_{d}^{q}\left( M,\mathcal{F}\right) \neq 0$, which is true if and only if $%
H_{d}^{\ast }\left( M,\mathcal{F}\right) $ satisfies Poincar\'{e} duality.
We now prove the analogous result for our modified basic cohomology $%
\widetilde{H}^{\ast }\left( M,\mathcal{F}\right) $.

First, we observe that the basic projection $\kappa _{b}$ of the mean
curvature one form is always $\widetilde{d}$-exact, because%
\begin{equation*}
\widetilde{d}\left( -2\right) =\left( d-\frac{1}{2}\kappa _{b}\wedge \right)
\left( -2\right) =\kappa _{b}.
\end{equation*}%
Also, we have the following:

\begin{lemma}
A Riemannian foliation $\left( M,\mathcal{F},g_{Q}\right) $ of codimension $%
q $ is taut, then $H_{d}^{\ast }\left( M,\mathcal{F}\right) \cong \widetilde{%
H}^{\ast }\left( M,\mathcal{F}\right) $. The converse is true if the
foliation is assumed to be transversally oriented.
\end{lemma}

\begin{proof}
If the foliation is taut, by \cite{Al} $\kappa _{b}=df$ for some basic
function $f$. Then $\widetilde{d}=d-\frac{1}{2}df\wedge =e^{\frac{f}{2}%
}\circ d\circ e^{-\frac{f}{2}}$. Thus $\left[ \alpha \right] \mapsto \left[
\exp \left( -\frac{f}{2}\right) \alpha \right] $ yields an isomorphism from $%
\widetilde{H}^{\ast }\left( M,\mathcal{F}\right) $ to $H_{d}^{\ast }\left( M,%
\mathcal{F}\right) $. Conversely, if $H_{d}^{\ast }\left( M,\mathcal{F}%
\right) \cong \widetilde{H}^{\ast }\left( M,\mathcal{F}\right) $, then
Poincar\'{e} duality is satisfied for the ordinary basic cohomology (from
the fact that it is satisfied for our twisted cohomology), which means $%
H_{d}^{q}\left( M,\mathcal{F}\right) \neq 0$. This is equivalent to the
tautness of the foliation if the foliation is assumed to be transversally
oriented.
\end{proof}


\begin{theorem}
\label{TautnessTheorem}A transversally oriented Riemannian foliation $\left(
M,\mathcal{F},g_{Q}\right) $ of codimension $q$ with bundle-like metric $%
g_{M}$ is taut if and only if $\widetilde{H}^{0}\left( M,\mathcal{F}\right)
\cong \widetilde{H}^{q}\left( M,\mathcal{F}\right) \neq 0$.
\end{theorem}

\begin{proof}
If $\left( M,\mathcal{F},g_{Q}\right) $ is taut, by the above Lemma we
conclude that 
\begin{equation*}
\widetilde{H}^{0}\left( M,\mathcal{F}\right) \cong \widetilde{H}^{q}\left( M,%
\mathcal{F}\right) \cong H_{d}^{q}\left( M,\mathcal{F}\right) \neq 0.
\end{equation*}

For the converse, we first assume that our metric is chosen so that the mean
curvature is basic-harmonic, meaning that $\kappa =\kappa _{b}$ and%
\begin{equation*}
\delta _{b}\kappa =\left( -\overline{\ast }d\overline{\ast }+\kappa
\lrcorner \right) \kappa =-\overline{\ast }d\overline{\ast }\kappa
+\left\vert \kappa \right\vert ^{2}=0.
\end{equation*}%
As we said before, it is always possible to choose the bundle-like metric
this way and our twisted basic cohomology groups are not affected. If $%
\widetilde{H}^{0}\left( M,\mathcal{F}\right) \neq 0$, there exists a
nontrivial basic function $h$ such that $\widetilde{d}h=0$. Hence $dh=\frac{1%
}{2}h\kappa _{b}=\frac{1}{2}h\kappa $. Then 
\begin{eqnarray*}
\Delta _{b}h &=&\delta _{b}dh=\frac{1}{2}\delta _{b}\left( h\kappa \right) =%
\frac{1}{2}\left( -\overline{\ast }d\overline{\ast }+\kappa \lrcorner
\right) \left( h\kappa \right) \\
&=&\frac{1}{2}\left( -\overline{\ast }\left( dh\wedge \right) \overline{\ast 
}\kappa -h\overline{\ast }d\overline{\ast }\kappa +h\left\vert \kappa
\right\vert ^{2}\right) \\
&=&\frac{1}{2}\left( -\frac{1}{2}h\overline{\ast }\left( \kappa \wedge
\left( \overline{\ast }\kappa \right) \right) -h\left\vert \kappa
\right\vert ^{2}+h\left\vert \kappa \right\vert ^{2}\right) \\
&=&-\frac{1}{4}h\left\vert \kappa \right\vert ^{2}.
\end{eqnarray*}%
%
%
%
%
%
%
%
%
%
%
The integral over $M$ on both sides yields $\left\vert \kappa \right\vert =0$%
. Thus, $\left( M,\mathcal{F}\right) $ is taut.
\end{proof}

\section{The basic signature operator}

\label{basicSignatureSection}
With notation as in Section \ref{Formula Section}, suppose that $\left( M,%
\mathcal{F},g_{Q}\right) $ is a transversally oriented Riemannian foliation
of even codimension $q$, and let $g_{M}$ be a specific compatible
bundle-like metric. Let 
\begin{equation*}
\bigstar =i^{k\left( k-1\right) +\frac{q}{2}}\overline{\ast }
\end{equation*}%
as an operator on basic $k$-forms, analogous to the involution used to
identify self-dual and anti-self-dual forms on a manifold. Note that this
endomorphism is symmetric, and 
\begin{equation*}
\bigstar ^{2}=1.
\end{equation*}

\begin{proposition}
We have \vspace{0in}$\bigstar \left( \widetilde{d}+\widetilde{\delta }%
\right) =-\left( \widetilde{d}+\widetilde{\delta }\right) \bigstar $. In
fact, $\bigstar \widetilde{d}=-\widetilde{\delta }\bigstar $ and $\bigstar 
\widetilde{\delta }=-\widetilde{d}\bigstar $.
\end{proposition}

\begin{proof}
By Proposition \ref{Formulas Proposition}, we have that, as an operator on
basic $k$-forms, 
\begin{eqnarray*}
\bigstar \left( \widetilde{d}+\widetilde{\delta }\right) &=&\bigstar 
\widetilde{d}+\bigstar \widetilde{\delta } \\
&=&i^{\left( k+1\right) \left( k\right) +\frac{q}{2}}\overline{\ast }%
\widetilde{d}+i^{\left( k-1\right) \left( k-2\right) +\frac{q}{2}}\overline{%
\ast }\widetilde{\delta } \\
&=&i^{\left( k+1\right) \left( k\right) +\frac{q}{2}}\left( -1\right) ^{k+1}%
\widetilde{\delta }\overline{\ast }+i^{\left( k-1\right) \left( k-2\right) +%
\frac{q}{2}}\left( -1\right) ^{k}\widetilde{d}\overline{\ast } \\
&=&i^{k^{2}+k+2k+2+\frac{q}{2}}\widetilde{\delta }\overline{\ast }%
+i^{k^{2}-3k+2+2k+\frac{q}{2}}\widetilde{d}\overline{\ast } \\
&=&i^{k\left( k-1\right) +4k+2+\frac{q}{2}}\widetilde{\delta }\overline{\ast 
}+i^{k\left( k-1\right) +2+\frac{q}{2}}\widetilde{d}\overline{\ast } \\
&=&-\left( \widetilde{\delta }+\widetilde{d}\right) i^{k\left( k-1\right) +%
\frac{q}{2}}\overline{\ast }=-\left( \widetilde{d}+\widetilde{\delta }%
\right) \bigstar .
\end{eqnarray*}%
From the computation, we see that $\bigstar \widetilde{d}=-\widetilde{\delta 
}\bigstar $ and $\bigstar \widetilde{\delta }=-\widetilde{d}\bigstar $.
\end{proof}

Let $\Omega^{+}\left( M,\mathcal{F}\right) $ denote the $+1$ eigenspace of $%
\bigstar $ in $\Omega^{\ast }\left( M,\mathcal{F}\right) $, and let $%
\Omega^{-}\left( M,\mathcal{F}\right) $ denote the $-1$ eigenspace of $%
\bigstar $ in $\Omega^{\ast }\left( M,\mathcal{F}\right) $. By the
proposition above, $D_{b}=\widetilde{d}+\widetilde{\delta }$ maps $%
\Omega^{\pm }\left( M,\mathcal{F}\right) $ to $\Omega^{\mp }\left( M,%
\mathcal{F}\right) $. Therefore, we may define the basic signature operator
as follows.

\begin{definition}
On a transversally oriented Riemannian foliation of even codimension, let
the \textbf{basic signature operator} be the operator $D_{b}:\Omega
^{+}\left( M,\mathcal{F}\right) \rightarrow \Omega ^{-}\left( M,\mathcal{F}%
\right) $. We define the \textbf{basic signature }$\sigma \left( M,\mathcal{F%
}\right) $ \textbf{of the foliation} to be the index%
\begin{equation*}
\sigma \left( M,\mathcal{F}\right) =\dim \ker \left( \left. \widetilde{%
\Delta }\right\vert _{\Omega ^{+}\left( M,\mathcal{F}\right) }\right) -\dim
\ker \left( \left. \widetilde{\Delta }\right\vert _{\Omega ^{-}\left( M,%
\mathcal{F}\right) }\right) .
\end{equation*}
\end{definition}

\begin{remark}
We note that such a definition is not possible for the operator $d+\delta
_{b}$, because the relationship in the proposition above does not hold for $%
d+\delta _{b}$.
\end{remark}

\begin{remark}
The basic signature was previously defined in \cite{EK2} and used in \cite%
{HatDj}, but in both cases the definitions and results hold only for the
special case of Riemannian foliations with minimal leaves. If the mean
curvature is zero, the reader may verify that $D_{b}$ and the signature $%
\sigma \left( M,\mathcal{F}\right) $ above coincide with the operator and
signature used in these other papers.
\end{remark}

\section{The twisted basic Laplacian, curvature, and tautness}

\label{Wein}

First, let $\chi _{\mathcal{F}}$ denote the characteristic form of any
oriented foliation $\mathcal{F}$ on a Riemannian manifold; this is the
leafwise volume form of the foliation, locally given by 
\begin{equation*}
\chi _{\mathcal{F}}=e_{1}^{\ast }\wedge ...\wedge e_{p}^{\ast },
\end{equation*}%
where $\left( e_{1},...,e_{p},e_{p+1},...,e_{n}\right) $ is a local
orthonormal frame of the $TM$ such that $\left( e_{1},...,e_{p}\right) $ is
a local orthonormal frame of $T\mathcal{F}$. Then Rummler's formula (see 
\cite{Rum}) gives%
\begin{equation}
d\chi _{\mathcal{F}}=-\kappa \wedge \chi _{\mathcal{F}}+\varphi _{0},
\label{RummlerFormula}
\end{equation}%
where $\varphi_{0}$ is a $\left( p+1\right) $-form on $M$ with the property
that $v_{1} \lrcorner v_{2}\lrcorner \cdots\lrcorner v_{p}\lrcorner \varphi
_{0}=0$ if $v_{j}\in T_{x}\mathcal{F}$, $1\leq j\leq p$, are leafwise
vectors.

We now determine relationships between the eigenvalues of the twisted basic
Laplacian $\widetilde{\Delta }$ and curvature. First we do a few
computations that will be useful later.

\begin{lemma}
\label{LieDerivativeFactsLemma}We have the following facts about the Lie
derivative. If $\alpha $ is any form, $V,W$ are vector fields, then

\begin{enumerate}
\item \label{CartanFormula}$\mathcal{L}_{V}=d\circ \left( V\lrcorner \right)
+\left( V\lrcorner \right) \circ d$

\item \label{AdjointCartanFormula}$\mathcal{L}_{V}^{\ast }=\left( V^{\flat
}\wedge \right) \circ \delta +\delta \circ \left( V^{\flat }\wedge \right) $

\item \label{LieDerivativeLeibnizRule}$\mathcal{L}_{V}\circ \left( \alpha
\wedge \right) =\left( \mathcal{L}_{V}\left( \alpha \right) \wedge \right)
+\left( \alpha \wedge \right) \circ \mathcal{L}_{V}$

\item \label{LieDerivIntProductLeibnizRule}$\mathcal{L}_{V}\circ \left(
W\lrcorner \right) =\left( \mathcal{L}_{V}\left( W\right) \lrcorner \right)
+\left( W\lrcorner \right) \circ \mathcal{L}_{V}$

\item \label{AdjointLieDerivLeibnizRule}$\mathcal{L}_{V}^{\ast }\circ \left(
\alpha \wedge \right) =\left( \alpha \wedge \right) \circ \mathcal{L}%
_{V}^{\ast }-\left( \left( \mathcal{L}_{V}\left( \alpha ^{\#}\right) \right)
^{\flat }\wedge \right) $ if $\alpha $ is a one-form.
\end{enumerate}
\end{lemma}

\begin{proof}
The first two facts follow from the Cartan formulas. The formulas after that
are standard (up to taking adjoints) and can be found in \cite{Lee}.
\end{proof}

\begin{lemma}
The operator $\left( \mathcal{L}_{V}+\mathcal{L}_{V}^{\ast }\right) $ is
zeroth order and thus commutes with multiplication by a function.
\end{lemma}

\begin{proof}
One may easily compute the commutator $\left[ \left( \mathcal{L}_{V}+%
\mathcal{L}_{V}^{\ast }\right) ,m_{f}\right] $, where $m_{f}$ denotes
multiplication by a function $f$, and the result is zero.
\end{proof}

\begin{lemma}
\label{OperatorAppl2WedgeProdLemma}If $\alpha $ is a one-form, $\beta $ is
any form, and $V$ is a vector field, then 
\[
\left( \mathcal{L}_{V}+\mathcal{L}%
_{V}^{\ast }\right) \left( \alpha \wedge \beta \right) =\alpha \wedge \left( 
\mathcal{L}_{V}+\mathcal{L}_{V}^{\ast }\right) \beta +\gamma \wedge \beta ,
\]
where%
\begin{eqnarray*}
\gamma &=&\mathcal{L}_{V}\left( \alpha \right) -\left( \mathcal{L}_{V}\left(
\alpha ^{\#}\right) \right) ^{\flat } \\
&=&\left( \mathcal{L}_{V}g\right) \left( \alpha ^{\#},\bullet \right) \\
&=&\nabla _{\alpha ^{\#}}V^{\flat }+\alpha \left( \nabla _{\bullet }V\right)
,
\end{eqnarray*}%
where $\mathcal{L}_{V}g$ is the Lie derivative of the metric tensor.
\end{lemma}

\begin{proof}
By Lemma \ref{LieDerivativeFactsLemma}, part (\ref{LieDerivativeLeibnizRule}%
), we have 
\begin{equation*}
\mathcal{L}_{V}\left( \alpha \wedge \beta \right) =\mathcal{L}_{V}\left(
\alpha \right) \wedge \beta +\alpha \wedge \mathcal{L}_{V}\left( \beta
\right) .
\end{equation*}%
Now using the part (\ref{AdjointLieDerivLeibnizRule}) of the same lemma,%
\begin{equation*}
\mathcal{L}_{V}^{\ast }\left( \alpha \wedge \beta \right) =\alpha \wedge 
\mathcal{L}_{V}^{\ast }\left( \beta \right) -\left( \mathcal{L}_{V}\left(
\alpha ^{\#}\right) \right) ^{\flat }\wedge \beta .
\end{equation*}%
The result with the first expression for $\gamma $ follows from adding the
two equations. Also, for any vector field $Z$,%
\begin{eqnarray*}
\gamma \left( Z\right) &=&\mathcal{L}_{V}\left( \alpha \right) \left(
Z\right) -\left( \mathcal{L}_{V}\left( \alpha ^{\#}\right) ,Z\right) \\
&=&V\left( \alpha ^{\#},Z\right) -\left( \alpha ^{\#},\mathcal{L}%
_{V}Z\right) -\left( \mathcal{L}_{V}\left( \alpha ^{\#}\right) ,Z\right) \\
&=&\left( \mathcal{L}_{V}g\right) \left( \alpha ^{\#},Z\right) ,
\end{eqnarray*}%
and also from the second line above we have%
\begin{eqnarray*}
\gamma \left( Z\right) &=&\left( \nabla _{V}\alpha ^{\#},Z\right) +\left(
\alpha ^{\#},\nabla _{V}Z\right) -\left( \alpha ^{\#},\mathcal{L}%
_{V}Z\right) -\left( \mathcal{L}_{V}\left( \alpha ^{\#}\right) ,Z\right) \\
&=&\left( \nabla _{\alpha ^{\#}}V,Z\right) +\left( \alpha ^{\#},\nabla
_{Z}V\right) .
\end{eqnarray*}
\end{proof}

\begin{corollary}
\label{oneFormCorollary}If $\alpha $ is a one-form and $V$ is a vector
field, then $\left( \mathcal{L}_{V}+\mathcal{L}_{V}^{\ast }\right) \alpha
=\left( \delta V^{\flat }\right) \alpha +\gamma $, with $\gamma $ given in
the previous Lemma, and we also have $\left( \left( \mathcal{L}_{V}+\mathcal{%
L}_{V}^{\ast }\right) \alpha ,\alpha \right) =\left( \delta V^{\flat
}\right) \left( \alpha ,\alpha \right) +2\left( \nabla _{\alpha
^{\#}}V,\alpha ^{\#}\right) $.
\end{corollary}

\begin{proof}
It follows from $\left( \mathcal{L}_{V}+\mathcal{L}_{V}^{\ast }\right)
\left( 1\right) =\delta V^{\flat }$.
\end{proof}

\begin{corollary}
\label{divergenceCorollary}Let $\alpha =\sum \alpha _{\tau }e^{\tau }$ be
any form, with $\tau =\left( \tau _{1},...,\tau _{k}\right) $ a multi-index
and $e^{\tau }=e^{\tau _{1}}\wedge ...\wedge e^{\tau _{k}}$. Then for any
vector field $V$, 
\begin{equation*}
\left( \mathcal{L}_{V}+\mathcal{L}_{V}^{\ast }\right) \left( \alpha \right)
=\left( \delta V^{\flat }\right) \alpha +\sum_{j=1}^{k}\alpha _{\tau
}e^{\tau _{1}}\wedge ...\wedge \gamma ^{j}\wedge ...\wedge e^{\tau _{k}},
\end{equation*}%
with $\gamma ^{j}=\left( \mathcal{L}_{V}g\right) \left( e_{\tau
_{j}},\bullet \right) $ replacing the $e^{\tau _{j}}$.
\end{corollary}

\begin{lemma}
\label{AdjLieDerivFormulaLemma}If $V$ is any vector field such that $%
dV^{\flat }=0$, then%
\begin{equation*}
\mathcal{L}_{V}^{\ast }=\mathcal{L}_{V}-2\nabla _{V}-\left( \delta V^{\flat
}\right)
\end{equation*}%
as operators on forms.
\end{lemma}

\begin{proof}
Choose a local orthonormal frame $\left( e_{j}\right) $ of $TM$, and let $%
\left( e^{j}\right) $ be the dual coframe. We assume that the frame is
chosen so that at the point in question, all $\nabla _{e_{j}}e_{k}$ vanish.
From Lemma \ref{LieDerivativeFactsLemma}, parts (\ref{CartanFormula}) and (%
\ref{AdjointCartanFormula}), we have, using the Einstein summation
convention,%
\begin{eqnarray*}
\mathcal{L}_{V}-\mathcal{L}_{V}^{\ast } &=&e^{j}\wedge \nabla _{e_{j}}\left(
V\lrcorner \right) +\left( V\lrcorner \right) e^{j}\wedge \nabla
_{e_{j}}+e_{j}\lrcorner \nabla _{e_{j}}\left( V^{\flat }\wedge \right)
+\left( V^{\flat }\wedge \right) e_{j}\lrcorner \nabla _{e_{j}} \\
&=&e^{j}\wedge \left( \nabla _{e_{j}}V\right) \lrcorner +e^{j}\wedge \left(
V\lrcorner \right) \nabla _{e_{j}}+\left( V\lrcorner \right) e^{j}\wedge
\nabla _{e_{j}} \\
&&+e_{j}\lrcorner \left( \nabla _{e_{j}}V^{\flat }\right) \wedge
+e_{j}\lrcorner \left( V^{\flat }\wedge \right) \nabla _{e_{j}}+\left(
V^{\flat }\wedge \right) e_{j}\lrcorner \nabla _{e_{j}}~.
\end{eqnarray*}%
Writing $V=V_{k}e_{k}$, we get%
\begin{eqnarray*}
\mathcal{L}_{V}-\mathcal{L}_{V}^{\ast } &=&e_{j}\left( V_{k}\right) \left(
e^{j}\wedge \right) \left( e_{k}\lrcorner \right) +V_{k}\left( e^{j}\wedge
\right) \left( e_{k}\lrcorner \right) \nabla _{e_{j}}+V_{k}\left(
e_{k}\lrcorner \right) \left( e^{j}\wedge \right) \nabla _{e_{j}} \\
&&+e_{j}\left( V_{k}\right) \left( e_{j}\lrcorner \right) \left( e^{k}\wedge
\right) +V_{k}\left( e_{j}\lrcorner \right) \left( e^{k}\wedge \right)
\nabla _{e_{j}}+V_{k}\left( e^{k}\wedge \right) \left( e_{j}\lrcorner
\right) \nabla _{e_{j}}~.
\end{eqnarray*}%
Since $dV^{\flat }=0$, therefore $e_{j}\left( V_{k}\right) =e_{k}\left(
V_{j}\right).$ Also we have that $\left( e^{j}\wedge \right) \left(
e_{k}\lrcorner \right) +\left( e_{k}\lrcorner \right) \left( e^{j}\wedge
\right) =\delta _{k}^{j}$. Thus, we find 
\begin{eqnarray*}
\mathcal{L}_{V}-\mathcal{L}_{V}^{\ast } &=&e_{j}\left( V_{j}\right)
+2V_{j}\nabla _{e_{j}} \\
&=&\delta V^{\flat }+2\nabla _{V}~.
\end{eqnarray*}
\end{proof}

Now we will use the above computations to establish the Weitzenb\"{o}%
ck-Bochner formula:

\begin{proposition}
\label{WeitzenbockProposition}(\textbf{Weitzenb\"{o}ck-Bochner formula for
the twisted basic Laplacian}) Let the bundle-like metric of a Riemannian
foliation $\left( M,\mathcal{F}\right) $ be chosen so that the mean
curvature form $\kappa$ is basic-harmonic. Then for any basic form $\alpha
\in \Omega \left( M,\mathcal{F}\right) $,%
\begin{equation*}
\widetilde{\Delta}\alpha =\nabla ^{\ast }\nabla \alpha +\rho \left( \alpha
\right) +\frac{1}{4}|\kappa|^{2}\alpha ,
\end{equation*}%
where $\rho \left( \alpha \right) =\sum_{i,j}e^{j}\wedge e_{i}\lrcorner
R\left( e_{i},e_{j}\right) \alpha $, with $R$ the transversal Riemann
curvature operator, and the sum is over a local orthonormal frame $%
\{e_j\}_{j=1,\cdots,q}$ of $Q$. In a particular case where $\alpha $ is a
one form, 
\begin{equation*}
\left\langle \widetilde{\Delta}\alpha ,\alpha \right\rangle =\left\langle
\nabla ^{\ast }\nabla \alpha ,\alpha \right\rangle +\mathrm{Ric}\left(
\alpha ^{\#},\alpha ^{\#}\right) +\frac{1}{4} |\kappa|^{2}|\alpha|^2,
\end{equation*}%
where $\mathrm{Ric}$ is the transversal Ricci curvature.
\end{proposition}

\begin{proof}
Since $\kappa $ is basic-harmonic, from \cite[Proposition 2.4]{PaRi}, we get 
\begin{eqnarray*}
0 &=&\delta _{b}\kappa =P\delta \kappa \\
&=&\delta \kappa +\left( P\kappa -\kappa \right) \lrcorner \kappa -\varphi
_{0}\lrcorner \left( \chi _{\mathcal{F}}\wedge \kappa \right) \\
&=&\delta \kappa ,
\end{eqnarray*}%
with $\varphi _{0}$ and $\chi _{\mathcal{F}}$ from Rummler's Formula (\ref%
{RummlerFormula}). Next, we write $\delta _{b}=\delta _{T}+\kappa \lrcorner
=-e_{j}\lrcorner \nabla _{e_{j}}+\kappa \lrcorner $, where we use the
Einstein summation convention. We have%
\begin{equation*}
\Delta _{b}=\left( d\delta _{T}+\delta _{T}d\right) +\mathcal{L}_{H}.
\end{equation*}%
Let $\left( e^{j}\right) $ be the dual coframe corresponding to $\left(
e_{j}\right) $. Suppose that we have chosen the frame so that $\nabla
_{e_{i}}e_{j}=0$ at the point in question. For any basic form $\alpha $,%
\begin{eqnarray*}
\left( d\delta _{T}+\delta _{T}d\right) \alpha &=&-e^{j}\wedge \nabla
_{e_{j}}\left( e_{i}\lrcorner \nabla _{e_{i}}\alpha \right) -e_{i}\lrcorner
\nabla _{e_{i}}\left( e^{j}\wedge \nabla _{e_{j}}\alpha \right) \\
&=&-e^{j}\wedge e_{i}\lrcorner \nabla _{e_{j}}\nabla _{e_{i}}\alpha
+e^{j}\wedge e_{i}\lrcorner \nabla _{e_{i}}\nabla _{e_{j}}\alpha -\nabla
_{e_{i}}\nabla _{e_{i}}\alpha \\
&=&e^{j}\wedge e_{i}\lrcorner R\left( e_{i},e_{j}\right) \alpha +\nabla
^{\ast }\nabla \alpha -\nabla _{H}\alpha \\
&=&\rho \left( \alpha \right) +\nabla ^{\ast }\nabla \alpha -\nabla
_{H}\alpha.
\end{eqnarray*}%
Here we have used the fact that for the adapted frame, 
\begin{equation*}
\nabla ^{\ast }\nabla =\nabla _{e_{j}}^{\ast }\nabla _{e_{j}}=\left( -\nabla
_{e_{j}}+\left( H,e_{j}\right) \right) \nabla _{e_{j}}=-\nabla
_{e_{i}}\nabla _{e_{i}}+\nabla _{H}~.
\end{equation*}%
Then%
\begin{equation*}
\Delta _{b}\alpha =\rho \left( \alpha \right) +\nabla ^{\ast }\nabla \alpha +%
\mathcal{L}_{H}\alpha -\nabla _{H}\alpha ,
\end{equation*}%
and 
\begin{eqnarray*}
\widetilde{\Delta}\alpha &=&\Delta _{b}\alpha -\frac{1}{2}\left( \mathcal{L}%
_{H}+\mathcal{L}_{H}^{\ast }\right) \alpha +\frac{1}{4}|\kappa|^{2}\alpha \\
&=&\nabla ^{\ast }\nabla \alpha +\rho \left( \alpha \right) +\frac{1}{2}%
\left( \mathcal{L}_{H}-\mathcal{L}_{H}^{\ast }\right) \alpha +\frac{1}{4}%
|\kappa|^{2}\alpha -\nabla _{H}\alpha.
\end{eqnarray*}%
With the help of Lemma \ref{AdjLieDerivFormulaLemma}, we have 
\begin{equation*}
\frac{1}{2}\left( \mathcal{L}_{H}-\mathcal{L}_{H}^{\ast }\right) -\nabla
_{H}=\delta \kappa =0.
\end{equation*}%
The result follows.
\end{proof}

\begin{corollary}
\label{NontautImpliesTrivialTwistedCohomThm}Suppose that $\left( M,\mathcal{F%
}\right) $ is a Riemannian foliation on a connected manifold $M$. Suppose
that the bundle-like metric is chosen so that $\kappa $ is basic-harmonic.
If the operator $\rho +\frac{1}{4}|\kappa |^{2}$ on $r$-forms is strictly
positive, then the twisted basic cohomology group $\widetilde{H}^{r}\left( M,%
\mathcal{F}\right) $ is trivial.
\end{corollary}

Since the basic Euler characteristic and basic signature may be computed
using the dimensions of $\ker \widetilde{\Delta }$, we have the following
result.

\begin{corollary}
\label{EulerSignatureCorollary} 
Under the same hypothesis as in Corollary \ref%
{NontautImpliesTrivialTwistedCohomThm} for all $r$ such that $0\leq r\leq q$%
, the basic Euler characteristic and basic signature are zero. Thus the
foliation is nontaut.
\end{corollary}

\begin{remark}
This fact for the basic Euler characteristic could be deduced from the Hopf
index theorem for Riemannian foliations (\cite{BePaRi}), but only in the
case where the basic mean curvature never vanishes. In that case, the dual
vector field is a basic normal vector field to the foliation that never
vanishes and thus yields $\chi \left( M,\mathcal{F}\right) =0$.
\end{remark}

\begin{corollary}
Suppose that the transversal Ricci curvature satisfies $Ric\left( X,X\right)
\geq 0$ for all vectors $X\in \Gamma(Q).$ If $M$ is nontaut, then $%
\widetilde{H}^{1}\left( M,\mathcal{F}\right) \cong \left\{ 0\right\} $.
\end{corollary}

\begin{corollary}
Suppose that the transversal Ricci curvature satisfies $Ric\left( X,X\right)
\geq 0$ for all vectors $X$ orthogonal to the Riemannian foliation $\left( M,%
\mathcal{F}\right) $ and $Ric\left( \bullet ,\bullet \right)>0$ for at least
one point of $M$. Then $\widetilde{H}^{1}\left( M,\mathcal{F}\right) \cong
\left\{ 0\right\} $.
\end{corollary}

\begin{corollary}
Suppose that the transversal sectional curvatures of $\left( M,\mathcal{F}%
\right) $ are nonnegative. If the foliation is nontaut, then the twisted
basic cohomology is identically zero, and thus the basic Euler
characteristic and signature are zero.
\end{corollary}

\begin{corollary}
Suppose that the transversal sectional curvatures are nonnegative and are
all positive for at least one point of $M$. If the foliation is nontaut,
then $\widetilde{H}^{r}\left( M,\mathcal{F}\right) \cong \left\{ 0\right\} $
for $1<r<q$.
\end{corollary}

\begin{remark}
Note that the curvature bounds above are weaker than those required by
previous results in \cite{Heb}, \cite{Ju}, etc.
\end{remark}

Using the Weitzenb\"{o}ck-Bochner formula, we can give a direct proof of
Hebda's result \cite{Heb}.

\begin{theorem}
\label{Heb} Let $M$ be a compact Riemannian manifold endowed with a
Riemannian foliation. If the transversal Ricci curvature is positive, then $%
H_B^1(M,\mathcal{F})=0.$
\end{theorem}

\begin{proof}
Since the basic cohomology groups are independent of the choice of the
bundle-like metric, we may assume that the mean curvature $\kappa $ is
basic-harmonic. Let $\alpha $ be a basic one-form closed and coclosed, i.e. $%
d\alpha =0$ and $\delta _{b}\alpha =0$. Then we find $\widetilde{d}\alpha =-%
\frac{1}{2}\kappa \wedge \alpha $ and $\widetilde{\delta }_{b}(\alpha )=-%
\frac{1}{2}\kappa \lrcorner \alpha .$ Thus $|\widetilde{d}\alpha |^{2}+|%
\widetilde{\delta }_{b}(\alpha )|^{2}=\frac{1}{4}|\kappa |^{2}|\alpha |^{2}.$
With the use of the Weitzenb\"{o}ck formula, we have 
\begin{equation*}
\int_{M}(\widetilde{\Delta }\alpha ,\alpha )=\frac{1}{4}\int_{M}|\kappa
|^{2}|\alpha |^{2}=\int_{M}|\nabla \alpha |^{2}+\int_{M}\mathrm{Ric}(\alpha
,\alpha )+\frac{1}{4}\int_{M}|\kappa |^{2}|\alpha |^{2}.
\end{equation*}%
Under the curvature assumption, we deduce that $\alpha =0$.
\end{proof}

\begin{corollary}
\label{strongerHebdaCorollary}Let $(M,\mathcal{F})$ be a Riemannian
foliation of a compact manifold and suppose that the transversal Ricci
curvature satisfies $\mathrm{Ric}(X,X)\geq 0$ for all $X\in \Gamma Q$ and $%
\mathrm{Ric}(X_{p},X_{p})>0$ for all nonzero $X_{p}\in \Gamma _{p}Q$ at one
point $p\in M$. Then $H_{B}^{1}(M,\mathcal{F})=0.$
\end{corollary}

\begin{proof}
With the weaker hypothesis, $\mathrm{Ric}(X,X)>0$ for all unit normal
vectors $X$ to the foliation on a neighborhood of $p$. If $\alpha $ is a
closed and coclosed basic one-form, by the previous proof $\alpha $ is zero
on that neighborhood. Since $\left( d+\delta _{b}\right) \alpha =0$, by \cite%
[Proposition 2.4]{PaRi} we have%
\begin{equation*}
\left( d+P\delta \right) \alpha =\left( d+\delta -\varphi _{0}\lrcorner \chi
_{\mathcal{F}}\wedge \right) \alpha =\left( d+\delta \right) \alpha =0,
\end{equation*}%
where $P:L^{2}\left( \Omega \left( M\right) \right) \rightarrow L^{2}\left(
\Omega \left( M,\mathcal{F}\right) \right) $ is the orthogonal projection, $%
\delta $ is the ordinary $L^{2}$ adjoint of $d$ on all forms, and $\varphi
_{0}$, $\chi _{\mathcal{F}}$ are from Rummler's formula (\ref{RummlerFormula}%
). The operator $d+\delta $ is a linear, first order elliptic operator that
satisfies the weak unique continuation property (see \cite{BooL}, \cite{Bar}%
, \cite{Ar}, \cite{Cor}). This means that since $\alpha $ is zero on an open
set, it is identically zero on all of $M$.
\end{proof}

We also find a direct proof for the following theorem established in \cite%
{ElKacSerg} (see also \cite{KT3}, \cite{KTduality}).

\begin{theorem}
\label{EK}Let $M$ be a compact, connected manifold endowed with a Riemannian
foliation $\mathcal{F}$. The top-dimensional basic cohomology is either
isomorphic to $0$ or $\mathbb{R}$.
\end{theorem}

\begin{proof}
Let $\alpha $ be a basic $q$-form closed and coclosed. Since $\alpha =f\nu $
where $\nu $ is the transverse volume form of the foliation and $f$ a basic
real-valued function on $M$, the term $(\rho (\alpha ),\alpha )=f^{2}(\rho
(\nu ),\nu )$ is equal to zero by the fact that $\nu $ is parallel and $f$
is a function. Now applying the Weitzenb\"{o}ck-Bochner formula to $\alpha $
gives that $\alpha $ is parallel, which means that $f$ is constant. If $f$
is always equal to zero, the basic cohomology is zero; otherwise it is
isomorphic to $\mathbb{R}.$
\end{proof}

In the following, we prove that the spectrum of the basic Laplacian is the
same as the twisted one under a curvature assumption.

\begin{proposition}
Let $M$ be a compact manifold endowed with a Riemannian foliation $\mathcal{F%
}$ with strictly positive transversal curvature. Then $\mathrm{spec}%
(\widetilde\Delta)=\mathrm{spec}(\Delta_b).$
\end{proposition}

\begin{proof}
By the Mason result \cite{Ma}, any bundle-like metric can be dilated to
another one $\bar{g}$ with basic-harmonic mean curvature $\bar{\kappa}$ and
with the same basic Laplacian. Since the transversal curvature is positive,
the first cohomology group is zero and hence $\bar{\kappa}=0$. In this case,
the operator $\Delta _{b}^{\bar{g}}=\widetilde{\Delta }^{\bar{g}}$. Using
the fact that the spectrum of $\widetilde{\Delta }^{\bar{g}}$ remains the
same for any possible change of the bundle-like, we deduce the proof of the
proposition.
\end{proof}

\begin{proposition}
\label{codim2Corollary}Let $\left( M,\mathcal{F}\right) $ be a Riemannian
foliation of codimension $2$ on a connected manifold. If the foliation is
nontaut, the basic cohomology groups satisfy $H_d^{0}\left( M,\mathcal{F}%
\right) =\ H_d^{1}\left( M,\mathcal{F}\right) =\mathbb{R}$, $H_d^{2}\left( M,%
\mathcal{F}\right) =\left\{ 0\right\} $. In all other cases, the foliation
is taut. Also, the twisted basic cohomology $\widetilde{H}^{\ast }\left( M,%
\mathcal{F}\right) $ is identically zero if and only if the foliation is
nontaut.
\end{proposition}

\begin{proof}
Suppose that the foliation is nontaut, so that $H_{d}^{0}\left( M,\mathcal{F}%
\right) \cong \mathbb{R}$, $H_{d}^{2}\left( M,\mathcal{F}\right) =\left\{
0\right\} $. Let $\kappa $ be chosen to be basic harmonic as in Corollary %
\ref{basicHarmonicChoiceCorollary}. Since $\left( d+\delta _{b}\right)
\kappa =0$, by the same argument as in Corollary \ref{strongerHebdaCorollary}%
, the weak unique continuation property implies that if $\kappa $ were zero
on an open set, it would be identically zero on $M$. Since $\left[ \kappa %
\right] \in H_{d}^{1}\left( M,\mathcal{F}\right) $ is nontrivial, the set on
which $\kappa $ is nonzero is open and dense. Furthermore, by \cite{Bar} we
know that the zero set of $\kappa $ is codimension two or more. By Rummler's
Theorem (\cite{Rum}), not all leaves are closed. Since $M$ is connected, the
principal stratum of the foliation must be saturated by noncompact leaves.
In the case where there are no compact leaves, the zero set of $\kappa $ is
one or less, so that $\kappa $ would have to be nonzero. Then, by the basic
Hopf index theorem \cite{BePaRi}, the basic Euler characteristic would have
to be zero, so that 
$\ H_{d}^{1}\left( M,\mathcal{F}\right) \cong \mathbb{R}$. 

The only remaining case is where $\kappa $ is zero at a finite number of
isolated closed leaves. If we use $H=\kappa ^{\#}$ as the basic normal
vector field in the basic Hopf index theorem \cite{BePaRi}, it remains to
calculate the sign of the determinant of the matrix $\left( a_{ij}\right)
=\left( \nabla _{e_{i}}H,e_{j}\right) $, i.e. the type of singularity of $H$
at each singular point. Since the leaf closures near $\kappa =0$ are
codimension one, the space of leaf closures looks locally like concentric
circles around the origin in $\mathbb{R}^{2}$. Because $\kappa $ is basic,
it must be either a source or a sink, which in both cases implies the index
of $H$ at the singular leaf is $1$. Since no orientation issues occur, we
see that the basic Euler characteristic is the number of singular leaves, a
positive number. On the other hand, 
\begin{eqnarray*}
\chi \left( M,\mathcal{F}\right) &=&\dim H_d^{0}\left( M,\mathcal{F}\right)
-\dim H_d^{1}\left( M,\mathcal{F}\right) +\dim H_d^{2}\left( M,\mathcal{F}%
\right) \\
&=&1-\dim H_d^{1}\left( M,\mathcal{F}\right) \leq 0.
\end{eqnarray*}%
We conclude that this last case cannot occur, and the basic Euler
characteristic of the foliation must be zero. The result follows.
\end{proof}

\begin{corollary}
\label{sigcodim2} The basic Euler characteristic and basic signature of a
nontaut Riemannian foliation of codimension two are zero.
\end{corollary}

\begin{remark}
In Section \ref{ExampleSection}, we show that it is possible to construct
nontaut Riemannian foliations of higher codimension with nonzero twisted
basic cohomology and nonzero twisted basic Euler characteristics.
\end{remark}






Recall that a group $G$ is \textbf{polycyclic} if there exists a finite
sequence of nested subgroups $1\vartriangleleft G_{1}\vartriangleleft
...\vartriangleleft G_{k}=G$ such that all factor groups are cyclic.

\begin{corollary}
\label{deformationCorollary}Suppose that $\left( M,\mathcal{F}\right) $ is a
nontaut Riemannian foliation of codimension two, and $\pi _{1}\left(
M\right) $ is polycyclic or has polynomial growth. Then the basic Euler
characteristic and basic signature are stable with respect to deformations
of $\left( M,\mathcal{F}\right) $ through continuous families of Riemannian
foliations, and in fact the dimensions of all basic cohomology groups are
also stable.
\end{corollary}

\begin{proof}
In \cite{Noz1}, Nozawa showed that nontautness is preserved in families of
Riemannian foliations on such manifolds. The two previous corollaries imply
the result.
\end{proof}

\begin{remark}
Note that in general the dimensions of basic cohomology groups are not
stable under such deformations; see \cite[Example 7.0.4]{Noz2} for a simple
example. However, $\dim H_d^{0}\left( M,\mathcal{F}\right) $ and $\dim
H_d^{q}\left( M,\mathcal{F}\right) $ are stable with respect to deformations
if $\pi _{1}\left( M\right) $ is polycyclic or has polynomial growth, as
implied by the discussion above.
\end{remark}





\begin{remark}
Because the twisted basic cohomology and ordinary basic cohomology groups
are independent of the choices of bundle-like metric and transverse
Riemannian structure (see Corollary \ref{indepTwBasicCohomCor}), we note
that the vanishing theorems in this section may be restated in terms of the
existence of bundle-like metrics with the required properties.
\end{remark}

\section{Examples\label{ExampleSection}}

\subsection{The Carri\`{e}re example}

We will compute the cohomology groups of the Carri\`{e}re example from \cite%
{Ca} in the $3$-dimensional case. Let $A$ be a matrix in $\mathrm{SL}_{2}(%
\mathbb{Z})$ of trace strictly greater than $2$. We denote respectively by $%
V_{1}$ and $V_{2}$ the eigenvectors associated with the eigenvalues $\lambda 
$ and $\frac{1}{\lambda }$ of $A$ with $\lambda >1$ irrational. Let the
hyperbolic torus $\mathbb{T}_{A}^{3}$ be the quotient of $\mathbb{T}%
^{2}\times \mathbb{R}$ by the equivalence relation which identifies $(m,t)$
to $(A(m),t+1)$. The flow generated by the vector field $V_{2}$ is a
transversally Lie foliation of the affine group. We denote by $K$ the
holonomy subgroup. The affine group is the Lie group $\mathbb{R}^{2}$ with
multiplication $(t,s).(t^{\prime },s^{\prime })=(t+t^{\prime },\lambda
^{t}s^{\prime }+s)$, and the subgroup $K$ is 
\begin{equation*}
K=\{(n,s),n\in \mathbb{Z},s\in \mathbb{R}\}.
\end{equation*}%
We choose the bundle-like metric (letting $\left( x,s,t\right) $ denote the
local coordinates in the $V_{2}$ direction, $V_{1}$ direction, and $\mathbb{R%
}$ direction, respectively) as 
\begin{equation*}
g=\lambda ^{-2t}dx^{2}+\lambda ^{2t}ds^{2}+dt^{2}.
\end{equation*}%
We will show that the twisted cohomology groups all vanish. First, we notice
that the mean curvature of the flow is $\kappa =\kappa _{b}=\log \left(
\lambda \right) dt$, since $\chi _{\mathcal{F}}=\lambda ^{-t}dx$ is the
characteristic form and $d\chi _{\mathcal{F}}=-\log \left( \lambda \right)
\lambda ^{-t}dt\wedge dx=-\kappa \wedge \chi _{\mathcal{F}}$. Since the flow
is nontaut, we have $\widetilde{H}^{0}\left( M,\mathcal{F}\right) \cong 
\widetilde{H}^{2}\left( M,\mathcal{F}\right) =0$ by Theorem \ref%
{TautnessTheorem}. We now show directly that $\widetilde{H}^{1}\left( M,%
\mathcal{F}\right) =0$ (allthough this is guaranted by Proposition \ref%
{codim2Corollary}). The $1$-forms $\alpha =dt$ and $\beta =\lambda ^{t}ds$
are left invariant. Every $K$-invariant $1$-form $\omega $ can be written as 
$\omega =f(t)\alpha +g(t)\beta $, where $f$ and $g $ are periodic functions.
For any $\widetilde{d}$-closed basic $1$-form $\omega,$ we have 
\begin{eqnarray*}
d\omega &=&\left( g^{\prime }(t)+g\log \left( \lambda \right) \right) \alpha
\wedge \beta \\
&=&\frac{1}{2}g\log \left( \lambda \right) \alpha \wedge \beta =\frac{1}{2}%
\kappa _{b}\wedge \omega .
\end{eqnarray*}%
We then deduce that $g^{\prime }=-\frac{1}{2}\log \left( \lambda \right) g$,
or $g=c\lambda ^{-\frac{t}{2}}$ for some $c\in \mathbb{R}$. Since $g$ is
periodic, it is zero. If $\omega $ is also $\widetilde{\delta }$-coclosed, 
\begin{eqnarray*}
\delta _{b}\omega &=&\delta _{b}(f\alpha )=-\alpha (f)+f\delta _{b}(\alpha )
\\
&=&-f^{\prime }(t)+f\log \left( \lambda \right) =\frac{1}{2}f\log \left(
\lambda \right) =\frac{1}{2}\kappa _{b}\lrcorner (f\alpha ).
\end{eqnarray*}%
The solution is again reduced to zero for periodic functions $f$. Thus, the
first twisted cohomology group is zero. \hfill $\square $

\subsection{Nontautness and nontrivial twisted cohomology}

In this example, the Riemannian foliation is nontaut, and the twisted basic
cohomology and basic Euler characteristic are nontrivial.

First, let $S^{1}=\mathbb{R}\diagup \mathbb{Z}$, and let $T^{2}=\mathbb{R}%
^{2}\diagup \mathbb{Z}^{2}$, with flat metrics to be chosen later. Consider
the manifold $X=\mathbb{R}\times _{\varphi }T^{2}$, a suspension of $T^{2}$
and a $T^{2}$ bundle over $S^{1}$, constructed using the identification:%
\begin{equation*}
\varphi \left( \widetilde{x},\left( a,b\right) \right) =\left( \widetilde{x}%
+1,\left( -a,-b\right) \right)
\end{equation*}%
for all $\widetilde{x}\in \mathbb{R}$, $\left( a,b\right) \in T^{2}=\mathbb{R%
}^{2}\diagup \mathbb{Z}^{2}$. We now exhibit a Riemannian foliation of $X$,
constructed as follows. First, observe that $\varphi$ is an
orientation-preserving isometry of $T^{2}$, for any given flat metric.
Observe that the lines in $T^{2}$ with slope $\frac{3+\sqrt{5}}{2}$
(parallel to one eigenvector of the matrix $\left( 
\begin{array}{cc}
1 & 1 \\ 
1 & 2%
\end{array}%
\right) $ ) are preserved by these isometries. For $b_{0}\in \mathbb{R}%
\diagup \mathbb{Z}$, the sets of the form%
\begin{equation*}
\widetilde{L_{b_{0}}}=\left\{ \left( \widetilde{x},\left( a,b\right) \right)
:\widetilde{x}\in \mathbb{R},a\in \mathbb{R}\diagup \mathbb{Z},b=\frac{3+%
\sqrt{5}}{2}a+b_{0}\right\} \subset \widetilde{X}=\mathbb{R}\times T^{2}
\end{equation*}%
form a Riemannian foliation $\widetilde{\mathcal{F}_{X}}$. Then the sets%
\begin{eqnarray*}
L_{b_{0}} &:&=\widetilde{L_{b_{0}}}\diagup \sim \\
\left( \widetilde{x},\left( a,b\right) \right) &\sim &\varphi \left( 
\widetilde{x},\left( a,b\right) \right)
\end{eqnarray*}%
form a Riemannian foliation $\mathcal{F}_{X}$ of the quotient $X=\mathbb{R}%
\times _{\varphi }T^{2}$ that is not transversally oriented, again for any
flat metrics. Note that $L_{b_{0}}=L_{b}$ for any $b$ in the orbit of $b_{0}$
via the action generated by $b\mapsto \frac{3+\sqrt{5}}{2}+b$, $b\mapsto -b$%
. Note that this Riemannian foliation $\mathcal{F}_{X}$ is dense in $X$, and
that it admits no basic vector fields or basic one-forms.

Next, let $Y$ be a surface of genus 2 with universal cover $\widetilde{Y}=%
\mathbb{H}$. Then $\pi _{1}\left( Y\right) $ is a group with presentation $%
\left\langle A,B,C,D:ABCDA^{-1}B^{-1}C^{-1}D^{-1}=1\right\rangle $. We
define the homomorphism%
\begin{equation*}
\widetilde{\psi }:\pi _{1}\left( Y\right) \rightarrow Diff\left( \widetilde{X%
},\widetilde{\mathcal{F}_{X}}\right)
\end{equation*}%
from $\pi _{1}\left( Y\right) $ to the group of foliated diffeomorphisms of $%
\left( \widetilde{X},\widetilde{\mathcal{F}_{X}}\right) $ defined by 
\begin{eqnarray*}
\widetilde{\psi }\left( A\right) \left( \widetilde{x},\left( a,b\right)
\right) &=&\left( \widetilde{x},\left( a+b,a+2b\right) \right) , \\
\psi \left( B\right) &=&\psi \left( C\right) =\psi \left( D\right) =\mathbf{1%
}\text{.}
\end{eqnarray*}%
Since $\widetilde{L_{b_{0}}}$ consists of lines parallel to one eigenvector
of $\left( 
\begin{array}{cc}
1 & 1 \\ 
1 & 2%
\end{array}%
\right) $, $\widetilde{\psi }\left( A\right) $ maps leaves of $\widetilde{%
\mathcal{F}_{X}}$ to themselves and commutes with the action $\varphi $, and
thus it descends to a homomorphism%
\begin{eqnarray*}
\psi &:&\pi _{1}\left( Y\right) \rightarrow Diff\left( X,\mathcal{F}%
_{X}\right) , \\
\psi \left( g\right) \left[ \left( \widetilde{x},\left( a,b\right) \right) %
\right] &:&=\left[ \widetilde{\psi }\left( g\right) \left( \widetilde{x}%
,\left( a,b\right) \right) \right] ,~g\in \pi _{1}\left( Y\right) .
\end{eqnarray*}%
Now we form the suspension%
\begin{eqnarray*}
M &=&\widetilde{Y}\times _{\psi }X \\
&=&\left\{ \left[ \left( \widetilde{y}^{\prime },x^{\prime }\right) :\left( 
\widetilde{y}^{\prime },x^{\prime }\right) =\left( g\widetilde{y},\psi
\left( g\right) x\right) \text{ for some }g\in \pi _{1}\left( Y\right) %
\right] :\left( \widetilde{y},x\right) \in \widetilde{Y}\times X\right\} ,
\end{eqnarray*}%
which is naturally endowed with the foliation $\mathcal{F}_{M}~$whose leaves
are equivalence classes $\left[ \left( \widetilde{y},L_{X}\right) \right]
\in \widetilde{Y}\times _{\psi }\mathcal{F}_{X}$ with $L_{X}\in \mathcal{F}%
_{X}$ (not unique in $L_{X}$). This foliation will be a Riemannian foliation
if we pull back any metric on $Y$ via $\widetilde{Y}\times _{\psi
}X\rightarrow Y$ and choose a metric on each fiber ($\cong X$) that is flat.
Note that we will need to modify the fiberwise metric as a function of $y\in
Y$ so that $\psi \left( g\right) $ acts by transverse isometries on the
fibers. Specifically, let $\lambda =\frac{3+\sqrt{5}}{2}$ be one eigenvalue
corresponding to the eigenvector $V_{1}=\left( 1,\frac{1+\sqrt{5}}{2}\right)
^{T}$ of $\left( 
\begin{array}{cc}
1 & 1 \\ 
1 & 2%
\end{array}%
\right) $, and let $V_{2}$ denote the other eigenvector corresponding to
eigenvalue $\lambda ^{-1}$. Choose a specific smooth closed curve $%
u\rightarrow \gamma _{A}\left( u\right) \in Y$ corresponding to $A\in \pi
_{1}\left( Y\right) $, with $\gamma _{A}\left( 0\right) =\gamma _{A}\left(
1\right) $. Letting $t_{1}$ and $t_{2}$ denote the (lifted) coordinates of $%
T^{2}$ corresponding to directions $V_{1}$ and $V_{2}$ respectively, $x$ the
coordinate on $S^{1}$, choose 
\begin{equation*}
ds^{2}=du^{2}+dx^{2}+\lambda ^{2u}dt_{1}^{2}+\lambda ^{-2u}dt_{2}^{2}
\end{equation*}%
to be the metric on the submanifold $\pi ^{-1}\left( \gamma _{A}\right) $,
where $\pi :M=\widetilde{Y}\times _{\psi }X\rightarrow Y$ is the projection.
Similarly choose metrics along paths $\gamma _{B}$, $\gamma _{C}$, $\gamma
_{D}$, but this time guaranteeing that the torus metrics on $\left(
t_{1},t_{2}\right) $ agree after traversing the circle (as well as on
intersections coming from the other curves). Then we extend the metric to a
metric on $M$ in any way so that it is fiberwise flat and that the metrics
on the horizontal submanifolds ($\widetilde{Y}$ parameter slices) are
pullbacks of metrics on $Y$. The resulting metric will be a bundle-like
metric for $\left( M,\mathcal{F}_{M}\right) $. The metric along the leaves
may then be modified so that the mean curvature form $\kappa $ is
basic-harmonic and is thus a harmonic one-form, and it is the pullback of a
one-form on $Y$. By doing a line integral along $\gamma _{A}$ we see that $%
\kappa $ determines a nontrivial class in $H^{1}\left( M\right) $, in fact
in $H^{1}\left( Y\right) $. Thus $\left( M,\mathcal{F}_{M}\right) $ is
nontaut.

By construction there are no basic forms except constants on $X$, and thus
every basic form on the codimension three foliation $\left( M,\mathcal{F}%
_{M}\right) $ is an element of $\Omega ^{\ast }\left( Y\right) $. Thus, the
ordinary basic cohomology groups are%
\begin{eqnarray*}
\dim H_{d}^{0}\left( M,\mathcal{F}_{M}\right) &=&1,~\dim H_{d}^{1}\left( M,%
\mathcal{F}_{M}\right) =4 \\
\dim H_{d}^{2}\left( M,\mathcal{F}_{M}\right) &=&1,~\dim H_{d}^{3}\left( M,%
\mathcal{F}_{M}\right) =0.
\end{eqnarray*}%
Then the basic Euler characteristic satisfies $\chi \left( M,\mathcal{F}%
_{M}\right) =-2$, so that the twisted basic cohomology groups are%
\begin{eqnarray*}
\dim \widetilde{H}^{0}\left( M,\mathcal{F}_{M}\right) &=&0,~\dim \widetilde{H%
}^{1}\left( M,\mathcal{F}_{M}\right) =\widetilde{h}^{1} \\
\dim \widetilde{H}^{2}\left( M,\mathcal{F}_{M}\right) &=&\widetilde{h}%
^{2},~\dim \widetilde{H}^{3}\left( M,\mathcal{F}_{M}\right) =0,
\end{eqnarray*}%
where $\widetilde{h}^{2}-\widetilde{h}^{1}=-2$. Note that Poincar\'{e}
duality is not satisfied, because $\left( M,\mathcal{F}_{M}\right) $ is not
transversally oriented. We have $\widetilde{h}^{2}\geq 0$, $\widetilde{h}%
^{1}\geq 2$.

\subsection{A transversally oriented example}

We now modify the previous example to produce a transversally oriented
Riemannian foliation that is nontaut and has nontrivial twisted basic
cohomology.

First, let $N$ be the connected sum of two copies of $S^{1}\times S^{2}$,
which has the property that $\pi _{1}\left( N\right) =\mathbb{Z}\ast \mathbb{%
Z}$, the free group on two generators $\alpha _{1}$ and $\alpha _{2}$. Let $%
\widetilde{N}$ be the universal cover of $N$, on which $\pi _{1}\left(
N\right) $ acts by deck transformations. Let $T^{3}=\mathbb{R}^{3}\diagup 
\mathbb{Z}^{3}$, with a flat metric to be specified later. Choose $\eta \in 
\mathbb{R}\setminus \mathbb{Q}$ . Consider the manifold 
\begin{eqnarray*}
X &=&\widetilde{N}\times _{\varphi }T^{3}=\widetilde{N}\times T^{3}\diagup
\sim \\
\left( \widetilde{x},\left( a,b,c\right) \right) &\sim &\left( \beta 
\widetilde{x},\varphi \left( \beta \right) \left( a,b,c\right) \right)
,~\beta \in \pi _{1}\left( N\right) ,
\end{eqnarray*}
a suspension of $T^{3}$ and a $T^{3}$ bundle over $N$, constructed using the
homomorphism $\varphi :\pi _{1}\left( N\right) \rightarrow $\textrm{Isom}$%
\left( T^{3}\right) $ generated by 
\begin{eqnarray*}
\varphi \left( \alpha _{1}\right) \left( a,b,c\right) &=&\left(
-a,-b,-c\right) \\
\varphi \left( \alpha _{2}\right) \left( a,b,c\right) &=&\left( a,b,c+\eta
\right)
\end{eqnarray*}%
for all $\left( a,b,c\right) \in T^{3}=\mathbb{R}^{3}\diagup \mathbb{Z}^{3}$%
. We now exhibit a Riemannian foliation of $X$, constructed as follows.
First, observe that each $\varphi \left( \beta \right) $ is an isometry of $%
T^{3}$, for any given flat metric. Observe that the lines in $T^{3}$
parallel to $\left( 1,\frac{1+\sqrt{5}}{2},0\right) ^{T}$, one eigenvector
of the matrix $B=\left( 
\begin{array}{ccc}
1 & 1 & 0 \\ 
1 & 2 & 0 \\ 
0 & 0 & 1%
\end{array}%
\right) $, are preserved by these isometries. For $\left( b_{0},c_{0}\right)
\in \mathbb{R}^{2}\diagup \mathbb{Z}^{2}$, the sets of the form%
\begin{equation*}
\widetilde{L_{b_{0},c_{0}}}=\left\{ \left( \widetilde{x},\left(
a,b_{0},c_{0}\right) \right) :\widetilde{x}\in \widetilde{N},a\in \mathbb{R}%
\diagup \mathbb{Z},b=\frac{3+\sqrt{5}}{2}a+b_{0}\right\} \subset X^{\prime }=%
\widetilde{N}\times T^{3}
\end{equation*}%
form a Riemannian foliation $\widetilde{\mathcal{F}_{X}}$. Then the sets%
\begin{equation*}
L_{b_{0},c_{0}}:=\widetilde{L_{b_{0},c_{0}}}\diagup \sim
\end{equation*}%
form a Riemannian foliation $\mathcal{F}_{X}$ of the quotient $X=\widetilde{N%
}\times _{\varphi }T^{3}$ that is transversally oriented, again for any flat
metrics. (We see that $\varphi \left( \alpha _{1}\right) $, although
orientation-reversing as a map from $T^{3}$ to itself, is transversally
orientation preserving.) Note that $L_{b_{0},c_{0}}=L_{b,c}$ for any $\left(
b,c\right) $ in the orbit of $\left( b_{0},c_{0}\right) $ via the action
generated by $b\mapsto \frac{3+\sqrt{5}}{2}+b$, $\left( b,c\right) \mapsto
\left( -b,-c\right) $, $c\mapsto c+\eta $. Note that this codimension two
Riemannian foliation $\mathcal{F}_{X}$ is dense in $X$, and that it admits
no basic vector fields or basic one-forms. The only basic forms for this
foliation are the constant functions and constant multiples of the
transverse volume form.

Next, let $Y$ be a surface of genus 2 with universal cover $\widetilde{Y}=%
\mathbb{H}$. Then $\pi _{1}\left( Y\right) $ is a group with presentation $%
\left\langle A,B,C,D:ABCDA^{-1}B^{-1}C^{-1}D^{-1}=1\right\rangle $. We
define the homomorphism%
\begin{equation*}
\widetilde{\psi }:\pi _{1}\left( Y\right) \rightarrow Diff\left( X^{\prime },%
\widetilde{\mathcal{F}_{X}}\right)
\end{equation*}%
from $\pi _{1}\left( Y\right) $ to the group of foliated diffeomorphisms of $%
\left( X^{\prime },\widetilde{\mathcal{F}_{X}}\right) $ defined by 
\begin{eqnarray*}
\widetilde{\psi }\left( A\right) \left( \widetilde{x},\left( a,b,c\right)
\right) &=&\left( \widetilde{x},\left( a+b,a+2b,c\right) \right) , \\
\psi \left( B\right) &=&\psi \left( C\right) =\psi \left( D\right) =\mathbf{1%
}\text{.}
\end{eqnarray*}%
Since $\widetilde{L_{b_{0},c_{0}}}$ consists of lines parallel to one
eigenvector of the matrix $B$, $\widetilde{\psi }\left( A\right) $ maps
leaves of $\widetilde{\mathcal{F}_{X}}$ to themselves and commutes with the
action $\varphi $, and thus it descends to a homomorphism%
\begin{eqnarray*}
\psi &:&\pi _{1}\left( Y\right) \rightarrow Diff\left( X,\mathcal{F}%
_{X}\right) , \\
\psi \left( g\right) \left[ \left( \widetilde{x},\left( a,b,c\right) \right) %
\right] &:&=\left[ \widetilde{\psi }\left( g\right) \left( \widetilde{x}%
,\left( a,b,c\right) \right) \right] ,~g\in \pi _{1}\left( Y\right).
\end{eqnarray*}%
Now we form the suspension%
\begin{eqnarray*}
M &=&\widetilde{Y}\times _{\psi }X \\
&=&\left\{ \left[ \left( \widetilde{y}^{\prime },x^{\prime }\right) :\left( 
\widetilde{y}^{\prime },x^{\prime }\right) =\left( g\widetilde{y},\psi
\left( g\right) x\right) \text{ for some }g\in \pi _{1}\left( Y\right) %
\right] :\left( \widetilde{y},x\right) \in \widetilde{Y}\times X\right\} ,
\end{eqnarray*}%
which is naturally endowed with the foliation $\mathcal{F}_{M}~$whose leaves
are equivalence classes $\left[ \left( \widetilde{y},L_{X}\right) \right]
\in \widetilde{Y}\times _{\psi }\mathcal{F}_{X}$ with $L_{X}\in \mathcal{F}%
_{X}$ (not unique in $L_{X}$). This foliation will be a Riemannian foliation
if we pull back any metric on $Y$ via $\widetilde{Y}\times _{\psi
}X\rightarrow Y$ and choose a metric on each fiber ($\cong X$) that is
transversally flat. Note that we will need to modify the fiberwise metric as
a function of $y\in Y$ so that $\psi \left( g\right) $ acts by transverse
isometries on the fibers. Specifically, let $\lambda =\frac{3+\sqrt{5}}{2}$
be one eigenvalue corresponding to the eigenvector $V_{1}=\left( 1,\frac{1+%
\sqrt{5}}{2},0\right) ^{T}$ of $B$, let $V_{2}$ denote the other eigenvector
corresponding to eigenvalue $\lambda ^{-1}$, and let $e_{3}=\left(
0,0,1\right) ^{T}$. Choose a specific smooth closed curve $u\rightarrow
\gamma _{A}\left( u\right) \in Y$ corresponding to $A\in \pi _{1}\left(
Y\right) $, with $\gamma _{A}\left( 0\right) =\gamma _{A}\left( 1\right) $.
Letting $t_{1}$, $t_{2}$, $t_{3}$ denote the (lifted) coordinates of $T^{3}$
corresponding to directions $V_{1}$, $V_{2}$, $e_{3}$ respectively, $x$ the
coordinate on $N$, choose 
\begin{equation*}
ds^{2}=du^{2}+dx^{2}+\lambda ^{2u}dt_{1}^{2}+\lambda
^{-2u}dt_{2}^{2}+dt_{3}^{2}
\end{equation*}%
to be the metric on the submanifold $\pi ^{-1}\left( \gamma _{A}\right) $,
where $\pi :M=\widetilde{Y}\times _{\psi }X\rightarrow Y$ is the projection.
Similarly choose metrics along paths $\gamma _{B}$, $\gamma _{C}$, $\gamma
_{D}$, but this time guaranteeing that the torus metrics on $\left(
t_{1},t_{2},t_{3}\right) $ agree after traversing the circle (as well as on
intersections coming from the other curves. Then we extend the metric to a
metric on $M$ in any way so that it is fiberwise flat and that the metrics
on the horizontal submanifolds ($\widetilde{Y}$ parameter slices) are
pullbacks of metrics on $Y$. The resulting metric will be a bundle-like
metric for $\left( M,\mathcal{F}_{M}\right) $. The metric along the leaves
may then be modified so that the mean curvature form $\kappa $ is
basic-harmonic and is thus a harmonic one-form, and it is the pullback of a
one-form on $Y$. By doing a line integral along $\gamma _{A}$ we see that $%
\kappa $ determines a nontrivial class in $H^{1}\left( M\right) $, in fact
in $H^{1}\left( Y\right) $. Thus $\left( M,\mathcal{F}_{M}\right) $ is
nontaut.

By construction there are no basic forms except constants and constant
multiples of the transverse volume form $\nu _{X}$ on $X$, and thus every
basic form on the codimension four foliation $\left( M,\mathcal{F}%
_{M}\right) $ is of the form $\omega _{1}+\omega _{2}\wedge \nu _{X}$ with $%
\omega _{1},\omega _{2}\in \Omega ^{\ast }\left( Y\right) $. Thus, the
ordinary basic cohomology groups are%
\begin{eqnarray*}
\dim H_{d}^{0}\left( M,\mathcal{F}_{M}\right) &=&1,~\dim H_{d}^{1}\left( M,%
\mathcal{F}_{M}\right) =4 \\
\dim H_{d}^{2}\left( M,\mathcal{F}_{M}\right) &=&2,~\dim H_{d}^{3}\left( M,%
\mathcal{F}_{M}\right) =4,~\dim H_{d}^{4}\left( M,\mathcal{F}_{M}\right) =1.
\end{eqnarray*}%
Then the basic Euler characteristic satisfies $\chi \left( M,\mathcal{F}%
_{M}\right) =-4$, so that the twisted basic cohomology groups are%
\begin{eqnarray*}
\dim \widetilde{H}^{0}\left( M,\mathcal{F}_{M}\right) &=&0,~\dim \widetilde{H%
}^{1}\left( M,\mathcal{F}_{M}\right) =\widetilde{h}^{1} \\
\dim \widetilde{H}^{2}\left( M,\mathcal{F}_{M}\right) &=&\widetilde{h}%
^{2},~\dim \widetilde{H}^{3}\left( M,\mathcal{F}_{M}\right) =\widetilde{h}%
^{3}=\widetilde{h}^{1},\dim \widetilde{H}^{4}\left( M,\mathcal{F}_{M}\right)
=0,
\end{eqnarray*}%
where $\widetilde{h}^{2}-2\widetilde{h}^{1}=-4$. Since the mean curvature
form as a form on $Y$ agrees with the mean curvature form in the previous
example and because the basic one-forms are the same, we must have $%
\widetilde{h}^{3}=\widetilde{h}^{1}\geq 2$, $\widetilde{h}^{2}\geq 0$ .

\end{document}